\documentclass[11pt]{amsart}

\usepackage{amsmath,amssymb,amsfonts,textcomp,amsthm,xifthen,psfrag,graphicx,color,pgfplots}
\usepackage{ifthen}

\usepackage{pgfplotstable}

\usepackage{fullpage}

\newcommand{\eps}{\varepsilon}

\newtheorem{theorem}{Theorem}
\newtheorem{lemma}[theorem]{Lemma}
\newtheorem{cor}[theorem]{Corollary}

\newtheorem{remark}[theorem]{Remark}
\theoremstyle{definition} 

\DeclareMathOperator{\ext}{\mathcal{E}}

\newcommand{\dual}[2]{\langle#1\hspace*{.5mm},#2\rangle}
\newcommand{\ip}[2]{(#1\hspace*{.5mm},#2)}

\newcommand{\abs}[1]{\vert #1 \vert}
\newcommand{\norm}[3][]{#1\|#2#1\|_{#3}}

\def\enorm#1{|\hspace*{-.5mm}|\hspace*{-.5mm}|#1|\hspace*{-.5mm}|\hspace*{-.5mm}|}

\newcommand{\diam}{\mathrm{diam}}

\newcommand{\wat}{\widehat}
\def\pwnabla{\nabla_\TT}

\def\div{{\rm div\,}}
\def\pwdiv{ {\rm div}_{\TT}\,}
\def\pwlap{ {\Delta}_{\TT}\,}

\newcommand{\trace}{\gamma}

\def\uc{u^c}
\def\Omc{\Omega^c}

\newcommand{\BEC}{{\rm ca}}

\newcommand{\R}{\ensuremath{\mathbb{R}}}
\newcommand{\N}{\ensuremath{\mathbb{N}}}
\newcommand{\HH}{\ensuremath{\boldsymbol{H}}}

\newcommand{\nn}{{\boldsymbol{n}}}
\newcommand{\vv}{\ensuremath{\boldsymbol{v}}}
\newcommand{\ww}{\ensuremath{\boldsymbol{w}}}
\newcommand{\TT}{\ensuremath{\mathcal{T}}}

\newcommand{\cS}{\ensuremath{\mathcal{S}}}

\newcommand{\el}{\ensuremath{T}}

\newcommand{\OO}{\ensuremath{\mathcal{O}}}

\newcommand{\ssigma}{{\boldsymbol\sigma}}
\newcommand{\ttau}{{\boldsymbol\tau}}
\newcommand{\qq}{{\boldsymbol{q}}}
\newcommand{\uu}{\boldsymbol{u}}

\newcommand{\slo}{\mathcal{V}}
\newcommand{\slp}{\widetilde\slo}
\newcommand{\hyp}{\mathcal{W}}
\newcommand{\dlo}{\mathcal{K}}
\newcommand{\dlp}{\widetilde\dlo}
\newcommand{\adlo}{\mathcal{K}'}

\newcounter{constantsnumber}
\def\setc#1{
  \ifthenelse{\equal{#1}{poinc}}{C_{\rm edge}}{ 
   \refstepcounter{constantsnumber}
   \label{const#1}C_{\theconstantsnumber}}}

\def\c#1{
  \ifthenelse{\equal{#1}{poinc}}{C_{\rm edge}}{ 
    C_{\ref{const#1}}}}

\newcommand{\est}{\operatorname{est}}
\newcommand{\err}{\operatorname{err}}

\begin{document}

\title[DPG-BEM coupling for reaction diffusion]{Robust coupling of DPG and BEM for a
singularly perturbed transmission problem}
\date{\today}

\author{Thomas F\"uhrer}
\author{Norbert Heuer}
\address{Facultad de Matem\'{a}ticas,
Pontificia Universidad Cat\'{o}lica de Chile,
Vicku\~{n}a Mackenna 4860, Santiago, Chile}
\email{\{tofuhrer,nheuer\}@mat.uc.cl}

\thanks{{\bf Acknowledgment.} 
Supported by CONICYT through FONDECYT projects 1150056, 3150012,
        and Anillo ACT1118 (ANANUM)}

\keywords{Reaction dominated diffusion, transmission problem,
          DPG method with optimal test functions, boundary elements, coupling,
          ultra-weak formulation, Calder\'on projector}
\subjclass[2010]{65N30, 65N38}
\begin{abstract}
We consider a transmission problem consisting of a singularly perturbed reaction diffusion equation
on a bounded domain and the Laplacian in the exterior, connected through standard transmission conditions.
We establish a DPG scheme coupled with Galerkin boundary elements for its discretization,
and prove its robustness for the field variables in so-called \emph{balanced norms}.
Our coupling scheme is the one from [F\"uhrer, Heuer, Karkulik: On the coupling of DPG and BEM, Math. Comp., accepted
for publication, 2016],
adapted to the singularly perturbed case by using the scheme from
[Heuer, Karkulik: A robust DPG method for singularly perturbed reaction diffusion problems, arXiv:1509.07560].
Essential feature of our method is that optimal test functions have to be computed only locally.
We report on various numerical experiments in two dimensions.
\end{abstract}
\maketitle


\section{Introduction}

The robust control of field variables (in certain norms) is the driving force behind the development
of the \emph{discontinuous Petrov-Galerkin method with optimal test functions} (DPG method).
By design, the DPG method provides access to approximation errors in the \emph{energy norm}
\cite{DemkowiczG_11_ADM,DemkowiczG_11_CDP}, and robustness means that, for singularly perturbed problems,
the energy error controls the error in the field variables with constants that do not depend
on the singular perturbation parameter. This robustness has been demonstrated
for convection/reaction dominated diffusion problems on bounded domains,
cf.~\cite{DemkowiczH_13_RDM,BroersenS_14_RPG,ChanHBTD_14_RDM,BroersenS_15_PGD,HeuerK_RDM,MelenkX_SPhpFEM}.
In this paper we extend this technique to transmission problems in unbounded domains where the interior
part of the problem is singularly perturbed. We use boundary integral equations to represent the exterior
problem and present a coupled DPG-boundary element approximation that provides robust control of both
the field and interface variables.
We note, however, that in practice this robustness hinges on an accurate resolution of the optimal test 
functions in the discrete test space.

The boundary element method (BEM) is an established tool for solving homogeneous elliptic or wave problems
with constant coefficients in unbounded domains. The BEM is also very useful when dealing with transmission
problems where one can couple domain-based methods (like finite elements) with BEM, cf., e.g.,
\cite{JohnsonN_80_CBI,Costabel_88_SMC,CostabelS_88_CFE,BielakM_91_SFE,Sayas_09_VJN} to give some classical references.
Only recently we learned how to appropriately couple the DPG method with boundary elements. Our first approach
\cite{HeuerK_15_DPG} consisted in applying the DPG technology to the whole transmission problem, formulated
in variational form including boundary integral operators to deal with the unbounded part of the domain.
This strategy suffers from the fact that, for approximating functions whose support touches the interface,
the calculation (approximation) of optimal test functions is a non-local problem located on a strip along
the interface. This goes against the general DPG strategy of localizing these problems on individual elements.
In \cite{FuehrerHK_CDB} we presented a coupling procedure for DPG and BEM that does not suffer from this
disadvantage. There, problems for calculating optimal test functions are entirely local, as in the case
of standard DPG, the only difference being the presence of a non-trivial kernel on the right-hand side.
We presented various coupling variants, of least-squares and Galerkin types for different boundary
integral equations. The model problem considered there is the Poisson equation in a bounded domain and
the Laplacian in its complement, subject to standard transmission conditions. The analysis of some
of the coupling variants makes use of the fact that both operators (in the interior and the exterior) are
essentially identical. We also claimed that the fourth variant (using both equations of the Calder\'on identity)
can be applied to transmission problems where the operators are different. In this paper we support this
very claim.

Our hypothesis is that the DPG-BEM coupling scheme based on the full Calder\'on identity can be extended in
a robust way to singularly perturbed transmission problems. This is the case when the singular perturbation
is localized on a bounded domain and can be robustly formulated by the DPG method. As a model problem
we couple reaction dominated diffusion on a bounded domain with the Laplacian in the exterior.
The solutions to reaction dominated diffusion problems suffer from boundary layers when the source term
does not comply with the boundary condition and, for discontinuous sources, also exhibit interior layers.
Various methods are proposed in the literature dealing, e.g., with layer adapted meshes and/or balanced norms,
see \cite{LiN_98_UCF,XenophontosF_03_UAS,LinS_12_BFE,RoosS_15_CSB} to cite a few.
For overviews see also \cite{RoosST_08_RNM,Linss_10_LAM}. Corresponding transmission problems suffer
from layers at interfaces, cf.~\cite{MaghnoujiN_06_BLT}, and possibly in the interior for discontinuous
right-hand side functions.
Literature on the numerical analysis of singularly perturbed transmission problems is scarce. 
In \cite{NicaiseX_09_FEM}, Nicaise and Xenophontos study finite elements for a one-dimensional
transmission problem with different diffusion coefficients, and in \cite{NicaiseX_13_Chp} they
consider the $hp$-version of the FEM for a transmission problem in two dimensions. Here, the
key point is an asymptotic expansion of the solution, based on the assumption of smooth geometry
and analytic data.

In \cite{HeuerK_RDM} we developed a robust DPG scheme for reaction dominated diffusion. This scheme is
based on a specific ultra-weak variational formulation that comprises three field variables
(the original unknown $u$ and scaled unknowns replacing $\nabla u$ and $\Delta u$) and skeleton variables
(two trace and two flux variables on the skeleton of the mesh).
Our quest for robustness led us to a formulation that controls the field variables robustly in so-called
\emph{balanced norms}. These norms are stronger than the energy norm stemming from the problem.
For reaction dominated problems with standard boundary layers,
Lin and Stynes \cite{LinS_12_BFE} proved that these norms are balanced in the sense that
their different components are of the same order when the diffusion parameter $\eps\to 0$.
In this paper we show that the robust DPG scheme from \cite{HeuerK_RDM} can be coupled with the BEM
to provide a robust scheme for the corresponding transmission problem. As already mentioned,
we follow the coupling strategy from \cite{FuehrerHK_CDB}.
In this case, it turns out, that we have to scale the Calder\'on identity by $\eps^{-1/2}$.
We are then able to prove (Theorem~\ref{thm:main} in Section~\ref{sec:main})
uniform equivalence of the energy norm to a norm that consists in the balanced norms
of the field variables and scaled norms (by $\eps^{-1/2}$) of the interface variables
(trace and $\eps$ times the normal derivative) in standard trace spaces (of orders $\pm 1/2$). This uniform
equivalence is perturbed by differently scaled (by powers of $\eps$) fractional-order Sobolev norms
of skeleton variables. This perturbation is not specific to our coupling procedure but has appeared previously
in DPG results of PDEs. In fact, for both convection dominated and reaction dominated problems, skeleton variables
are not robustly controlled by the energy norm, see~\cite{DemkowiczH_13_RDM,HeuerK_RDM}.
In that cases, also the upper bounds for the energy norm suffer from slight $\eps$-perturbations of norms for
some skeleton variables. Whereas the loss of robust control of skeleton variables is not critical since they
are not of primary interest, the loss in upper bounds could be compensated by higher-order approximations
of the corresponding skeleton variable (see also Remark~\ref{rem:eff} below).
Anyway, as in standard DPG schemes our method delivers quasi-optimal
approximations in the energy norm (in our case the DPG-energy norm plus scaled norms for interface variables)
and this norm controls robustly the field variables in balanced norms and the interface variables
in the very norms that constitute the energy norm. This result is immediate consequence of Theorem~\ref{thm:main}
and is stated in Corollary~\ref{cor:bestapprox} (also in Section~\ref{sec:main}).

The paper is organized as follows. In Section~\ref{sec:main} we give a precise mathematical formulation of the
transmission problem, propose the DPG method coupled to BEM and state our main results (Theorem~\ref{thm:main} and
Corollary~\ref{cor:bestapprox}).
Section~\ref{sec:proof} deals with some technical results and presents a proof of Theorem~\ref{thm:main}.
Finally, in Section~\ref{sec:num}, we support our theoretical results with various numerical experiments in two dimensions.

Throughout the paper, suprema are taken over sets excluding the null element.
The notation $A\lesssim B$ is used to say that $A\leq C\cdot B$ with a constant $C>0$ which
does not depend on any quantities of interest. In particular, $C$ does not depend on the diffusion
coefficient $\eps>0$.
Correspondingly, the notation $A\gtrsim B$ is used,
and $A\simeq B$ means that $A\lesssim B$ and $B\lesssim A$.

\section{Formulation of coupling method and main results} \label{sec:main}
%
\subsection{Model problem}\label{sec:main:model}
We consider the following model transmission problem:
given a diffusion constant $0<\eps\leq1$ and some data $f\in L_2(\Omega)$, $u_0\in H^{1/2}(\Gamma)$, $\phi_0\in H^{-1/2}(\Gamma)$,
find $u\in H^1(\Omega)$ and $u^c\in H^1_{\mathrm{loc}}(\Omc)$ such that
\begin{subequations}\label{tp}
\begin{align}
  -\eps \Delta u + u &= f \text{ in } \Omega,\label{tpa}\\
  -\Delta \uc &= 0 \text{ in } \Omc, \label{tpb}\\
  u-\uc &= u_0 \text{ on } \Gamma,\label{tpc}\\
  \eps \frac{\partial u}{\partial\nn_\Omega}-\frac{\partial \uc}{\partial\nn_\Omega} &= \phi_0 \text{ on } \Gamma,\label{tpd}\\
  \uc(x) &= \begin{cases}
    b \log\abs{x} + \OO(\abs{x}^{-1})& d=2, \\
    \OO(\abs{x}^{-1}) & d=3,
  \end{cases} \text{ as } \abs{x}\rightarrow\infty.\label{tpe}
\end{align}
\end{subequations}
Here, $\Omega\subset\R^d$, $d\in\{2,3\}$ is a bounded, simply connected Lipschitz domain
with polygonal boundary $\Gamma$, $\Omc:=\R^d\setminus\overline\Omega$, and
normal vector $\nn_\Omega$ on $\Gamma$ pointing in direction of $\Omc$.
We use the standard Sobolev spaces $H^1(\Omega)$, $H^1_{\mathrm{loc}}(\Omc)$, trace space
$H^{1/2}(\Gamma)$ and its dual $H^{-1/2}(\Gamma)$, see Section~\ref{sec:sob} for precise definitions.
Note that for $d=2$, the behavior of $\uc$ at infinity involves some constant $b\in\R$.

\subsection{Sobolev spaces, discrete spaces, and norms}\label{sec:sob}
For a Lipschitz domain $\omega\subset\R^d$ we use the standard Sobolev spaces
$L_2(\omega)$, $H^1(\omega)$, $\HH(\div,\omega)$.
Vector-valued spaces and functions will be denoted by bold symbols.
Denoting by $\trace_\omega$ the trace operator acting on $H^1(\omega)$, we define the trace space
\begin{align*}
  H^{1/2}(\partial\omega) := \left\{ \trace_\omega u\,:\, u\in H^1(\omega) \right\}
  \quad\text{ and its dual }\quad H^{-1/2}(\partial\omega) := \bigl(H^{1/2}(\partial\omega)\bigr)',
\end{align*}
and use the canonical norms.
Here, duality is understood with respect to $L_2(\partial\omega)$ as a pivot space,
i.e., using the extended $L_2(\partial\omega)$ inner product $\ip{\cdot}{\cdot}_{\partial\omega}$.
The $L_2(\Omega)$ inner product will be denoted by $\ip{\cdot}{\cdot}$.
Let $\TT$ denote a disjoint partition of $\Omega$ into open Lipschitz sets $\el\in\TT$,
i.e., $\bigcup_{\el\in\TT}\overline\el = \overline\Omega$.
The set of all boundaries of all elements forms the skeleton
$\cS := \left\{ \partial\el \mid \el\in\TT \right\}$.
By $\nn_M$ we mean the outer normal vector on $\partial M$ for a Lipschitz set $M$.
On a partition $\TT$ we use the product spaces 
\begin{align*}
  H^1(\TT) &:= \{ w\in L_2(\Omega) \,:\, w|_T \in H^1(T) \,\forall T\in\TT\}, \\
  \HH(\div,\TT) &:= \{ \qq\in (L_2(\Omega))^d \,:\, \qq|_T \in \HH(\div,T) \,\forall T\in\TT\}, \\
  H^1(\Delta,\TT) &:= \{w\in H^1(\TT) \,:\, \Delta w|_T \in L_2(T) \,\forall T\in\TT\}.
\end{align*}
The symbols $\pwnabla$, $\pwdiv$, resp. $\pwlap$ denote, the $\TT$-piecewise gradient, divergence, resp. Laplace operators. 
On the skeleton $\cS$ of $\TT$ we introduce the trace spaces
\begin{align*}
  H^{1/2}(\cS) &:=
  \Big\{ \wat u \in \Pi_{\el\in\TT}H^{1/2}(\partial\el)\,:\,
         \exists w\in H^1(\Omega) \text{ such that } 
         \wat u|_{\partial\el} = w|_{\partial\el}\; \forall \el\in\TT \Big\},\\
  H^{-1/2}(\cS) &:=
  \Big\{ \wat\sigma \in \Pi_{\el\in\TT}H^{-1/2}(\partial\el)\,:\,
         \exists \qq\in\HH(\div,\Omega) \text{ such that } 
         \wat\sigma|_{\partial\el} = (\qq\cdot\nn_{\el})|_{\partial\el}\; \forall\el\in\TT \Big\}.
\end{align*}
These spaces are equipped with the norms inherited from the \textit{balanced} norms of their corresponding volume
spaces, i.e., 
\begin{subequations} \label{Hpm}
\begin{align}
  \norm{\wat u}{1/2,\cS} &:=
  \inf \left\{ (\norm{w}{L_2(\Omega)}^2 + \eps^{1/2} \norm{\nabla w}{L_2(\Omega)}^2)^{1/2} \,:\, w\in H^1(\Omega),
               \wat u|_{\partial\el}=w|_{\partial\el}\; \forall\el\in\TT \right\},\\
  \norm{\wat\sigma}{-1/2,\cS} &:=
  \inf \left\{ (\norm{\qq}{L_2(\Omega)}^2\!+\!\eps \norm{\div\qq}{L_2(\Omega)}^2)^{1/2} \,:\, \qq\!\in\!\HH(\div,\Omega), 
               \wat\sigma|_{\partial\el}\!=\!(\qq\cdot\nn_{\el})|_{\partial\el}\; \forall\el\!\in\!\TT \right\}.
\end{align}
\end{subequations}
For functions $\wat u\in H^{1/2}(\cS)$, $\wat\sigma\in H^{-1/2}(\cS)$
and $\ttau\in\HH(\div,\TT)$, $v\in H^1(\TT)$ we define
\begin{align*}
  \dual{\wat u}{\ttau\cdot\nn}_\cS
  := \sum_{\el\in\TT}\dual{\wat u|_{\partial\el}}{\ttau\cdot\nn_\el}_{\partial\el},\quad
  \dual{\wat\sigma}{v}_\cS
  := \sum_{\el\in\TT}\dual{\wat\sigma|_{\partial\el}}{v}_{\partial\el}.
\end{align*}
In particular, for functions in conforming spaces, i.e., $\ttau\in\HH(\div,\Omega)$, $v\in H^1(\Omega)$, we have
\begin{align*}
  \dual{\wat u}{\ttau\cdot\nn}_\cS = \dual{\wat u}{\ttau\cdot\nn_\Omega}_\Gamma \quad\text{and}\quad
  \dual{\wat\sigma}{v}_\cS = \dual{\wat\sigma}{v}_\Gamma.
\end{align*}
%

\subsection{DPG and ultra-weak variational formulation}\label{sec_fem}
In this section we recall the DPG formulation of~\cite{HeuerK_RDM}.
To this end, let $\TT$ be a partition of $\Omega$ with skeleton $\cS$.
Define the vector spaces
\begin{align*}
   \widetilde U &:= L_2(\Omega)\times (L_2(\Omega))^d \times L_2(\Omega)
   \times H^{1/2}(\cS)\times H^{1/2}(\cS)
   \times H^{-1/2}(\cS) \times H^{-1/2}(\cS), \\
   U &:= \{ (u,\ssigma,\rho,\wat u^a,\wat u^b, \wat \sigma^a,\wat\sigma^b)\in\widetilde U \,:\, 
   \wat u^a|_\Gamma = \wat u^b|_\Gamma \}, \\
   U_0 &:= \{ (u,\ssigma,\rho,\wat u^a,\wat u^b, \wat \sigma^a,\wat\sigma^b)\in\widetilde U \,:\, 
   \wat u^a|_\Gamma = 0 = \wat u^b|_\Gamma \} \subseteq U, \\
   V &:= H^1(\TT)\times\HH(\div,\TT) \times H^1(\Delta,\TT),
\end{align*}
where in $V$ we introduce the norm
\begin{align*}
  \norm{ (\mu,\ttau,v)}V^2 &:= \eps^{-1} \norm{\mu}{L_2(\Omega)}^2 + \norm{\pwnabla \mu}{L_2(\Omega)}^2 
  + \eps^{-1/2} \norm{\ttau}{L_2(\Omega)}^2 + \norm{\pwdiv \ttau}{L_2(\Omega)}^2 \\
  &\qquad + \norm{v}{L_2(\Omega)}^2 + \eps^{1/2} \norm{\pwnabla v}{L_2(\Omega)}^2 
  + \eps^{3/2} \norm{\pwlap v}{L_2(\Omega)}^2.
\end{align*}
This norm is induced by the inner product $\dual\cdot\cdot_V$.

Given the right-hand side $f\in L_2(\Omega)$, we rewrite the interior problem~\eqref{tpa} as
the first order system
\begin{align*}
  \rho -\div\ssigma = 0, \quad \eps^{-1/4}\ssigma -\nabla u = 0, \quad -\eps^{3/4}\rho +u = f.
\end{align*}
Testing these equations, respectively, with $\mu\in H^1(\TT)$, $\ttau\in \HH(\div,\TT)$, and
$v-\eps^{1/2}\pwlap v$ for $v\in H^1(\Delta,\TT)$, and integrating by parts,
this leads to the ultra-weak formulation:
Find $\uu = (u,\ssigma,\rho,\wat u^a,\wat u^b, \wat \sigma^a,\wat\sigma^b) \in U$ such that
\begin{subequations}\label{eq:ultraweak}
\begin{align}
  \ip{\rho}{\mu} + \ip{\ssigma}{\pwnabla \mu} - \dual{\wat\sigma^a}{\mu}_\cS & = 0, \\
  \eps^{-1/4}\ip{\ssigma}{\ttau} + \ip{u}{\pwdiv \ttau} - \dual{\wat u^a}{\ttau\cdot\nn}_\cS &= 0, \\
  \eps^{3/4} \ip{\ssigma}{\pwnabla v} - \eps^{3/4} \dual{\wat\sigma^b}{v}_\cS + \ip{u}{v} \hspace*{3.2cm} &\quad  \nonumber\\
    + \eps^{5/4} \ip{\rho}{\pwlap v} + \eps^{1/4} \ip{\ssigma}{\pwnabla v} 
    -\eps^{1/2}\dual{\wat u^b}{\pwnabla v\cdot\nn}_\cS
    &= \ip{f}{v-\eps^{1/2}\pwlap v} \label{eq:ultraweak:third}
\end{align}
\end{subequations}
holds for all $\vv = (\mu,\ttau,v)\in V$.
To see~\eqref{eq:ultraweak:third}, we test $-\eps^{3/4}\rho + u = f$ with $v\in H^1(\Delta,\TT)$, and use $\rho =
\div\ssigma$, leading to
\begin{align*}
  -\eps^{3/4} \ip{\div\ssigma}{v} +\ip{u}{v} = \ip{f}{v}.
\end{align*}
Then, integrating by parts gives the first line of~\eqref{eq:ultraweak:third}. 
The second line is obtained by testing $-\eps^{3/4}\rho + u = f$ with $-\eps^{1/4}\pwlap v$, integrating by parts and
using $\ssigma = \eps^{1/4}\nabla u$.
The left-hand side of~\eqref{eq:ultraweak} gives rise to the bilinear form, resp. functional
\begin{subequations}
\begin{align}
 \begin{split}
  b(\uu,\vv) &:= \ip{\rho}{\mu} + \ip{\ssigma}{\pwnabla \mu} - \dual{\wat\sigma^a}{\mu}_\cS + 
  \eps^{-1/4}\ip{\ssigma}{\ttau} + \ip{u}{\pwdiv \ttau} - \dual{\wat u^a}{\ttau\cdot\nn}_\cS \\
  &\qquad + \eps^{3/4} \ip{\ssigma}{\pwnabla v} - \eps^{3/4} \dual{\wat\sigma^b}{v}_\cS + \ip{u}{v}, \\
  &\qquad  + \eps^{5/4} \ip{\rho}{\pwlap v} + \eps^{1/4} \ip{\ssigma}{\pwnabla v} 
  -\eps^{1/2}\dual{\wat u^b}{\pwnabla v\cdot\nn}_\cS 
 \end{split} \\ 
  L_V(\vv) &:= \ip{f}{v-\eps^{1/2}\pwlap v}
\end{align}
\end{subequations}
for $\uu = (u,\ssigma,\rho,\wat u^a,\wat u^b, \wat \sigma^a,\wat\sigma^b) \in U$, $\vv = (\mu,\ttau,v)\in V$.

We define the (weighted) trial-to-test operator $\Theta_\beta : U \to V$ with (fixed) $\beta>0$ by the relations
\begin{align}\label{eq:ttop}
  \Theta_\beta = \beta\Theta \quad\text{and}\quad \dual{\Theta\uu}{\ww}_V = b(\uu,\ww) \quad\text{for all } \ww\in V.
\end{align}
For the analysis we make use of the strong form of~\eqref{eq:ultraweak} with operator $B : U\to V'$, i.e.,
\begin{align*}
  B\uu = L_V \quad\text{in } V'.
\end{align*}
We stress that $B$ has a non-trivial kernel due to missing boundary conditions.
These will be handled by the transmission conditions~\eqref{tpc}--\eqref{tpd} and the exterior
problem~\eqref{tpb} with radiation condition~\eqref{tpe}.
For convenience, we introduce the following restriction operators
\begin{align*}
\begin{aligned}
  \gamma &: U \to H^{1/2}(\Gamma)\times H^{-1/2}(\Gamma) ,&\quad 
  \gamma\uu &:= (\wat u^a|_\Gamma,\wat\sigma^a|_\Gamma), \\
  \gamma_0 &: U \to H^{1/2}(\Gamma), &\quad
  \gamma_0\uu &:= \wat u^a|_\Gamma,  \\
  \gamma_\nn^\star &: U \to H^{-1/2}(\Gamma), &\quad 
  \gamma_\nn^\star\uu &:= \wat \sigma^\star|_\Gamma \qquad\text{with } \star\in\{a,b\}.
\end{aligned}
\end{align*}
Recall that by definition of the space $U$ we have that $\wat u^a|_\Gamma = \wat u^b|_\Gamma$.
We directly obtain that these operators are bounded if we equip $H^{\pm 1/2}(\Gamma)$ with the norms
\begin{align*}
  \norm{\wat u}{1/2,\Gamma} &:=  \inf \left\{ (\norm{w}{L_2(\Omega)}^2 + \eps^{1/2} 
  \norm{\nabla w}{L_2(\Omega)}^2)^{1/2} \,:\, w\in H^1(\Omega),\,  \wat u=w|_\Gamma\right\},\\
  \norm{\wat\sigma}{-1/2,\Gamma} &:=
  \inf \left\{ (\norm{\qq}{L_2(\Omega)}^2 + \eps \norm{\div\qq}{L_2(\Omega)}^2)^{1/2} \,:\, \qq\in\HH(\div,\Omega), \,
               \wat\sigma = \qq\cdot\nn_\Omega|_\Gamma  \right\}.
\end{align*}
However, as we use boundary integral operators to tackle the exterior Laplace problem~\eqref{tpb},~\eqref{tpe}, 
we make use of the canonical trace norms (associated to the Laplacian)
\begin{align*}
  \norm{\wat u}{H^{1/2}(\Gamma)} &:= 
  \inf \left\{ \norm{w}{H^1(\Omega)}  \,:\, w\in H^1(\Omega),\,  \wat u=w|_\Gamma\right\},\\
  \norm{\wat\sigma}{H^{-1/2}(\Gamma)} &:=
  \inf \left\{ \norm{\qq}{\HH(\div,\Omega)} \,:\, \qq\in\HH(\div,\Omega),\,\wat\sigma = \qq\cdot\nn_\Omega |_\Gamma  \right\}.
\end{align*}
Still we have (non-uniform) boundedness of the restriction operators, since $\norm{\wat u}{H^{1/2}(\Gamma)} \leq \eps^{-1/4}
\norm{\wat u}{1/2,\Gamma}$ resp. $\norm{\wat\sigma}{H^{-1/2}(\Gamma)} \leq \eps^{-1/2} \norm{\wat\sigma}{-1/2,\cS}$.
Also note that $\norm{\wat\sigma}{H^{-1/2}(\Gamma)} \simeq \sup_{v\in H^{1/2}(\Gamma)} 
\dual{\wat\sigma}{v}_\Gamma/\norm{v}{H^{1/2}(\Gamma)}$.

\subsection{Boundary integral operators} \label{sec_BIO}
The exterior part of problem \eqref{tp} will be dealt with by boundary integral operators.
To this end we need some further definitions. The fundamental solution of the Laplacian is
\begin{align*}
  G(z) :=
    \begin{cases}
    -\frac{1}{2\pi}\log\abs{z}\quad& (d=2),\\
    \frac{1}{4\pi}\frac{1}{\abs{z}}& (d=3),
  \end{cases}
\end{align*}
and the corresponding single layer and double layer potentials are
\begin{align*}
  \slp\phi(x) := \int_\Gamma G(x-y)\phi(y)\,ds_y,\quad
  \dlp v(x) := \int_\Gamma \partial_{\nn_\Omega(y)}G(x-y)v(y)\,ds_y,
  \quad x\in \R^d\setminus\Gamma.
\end{align*}
Taking the trace and normal derivative leads to three boundary integral operators
that satisfy the relations
\begin{align}\label{slo_dlo}
  \begin{split}
  \slo = \trace_\Omega \slp,\qquad
  \dlo = 1/2 + \trace_\Omega\dlp,\qquad
  \hyp = -\partial_{\nn_\Omega} \dlp
  \qquad\text{on}\ \Gamma.
  \end{split}
\end{align}
They are the \textit{single layer}, \textit{double layer},
and \textit{hypersingular operators}, respectively. The adjoint operator of $\dlo$ is denoted by $\adlo$.
These operators are linear and bounded as mappings
$\slo: H^{-1/2}(\Gamma)\rightarrow H^{1/2}(\Gamma)$,
$\dlo: H^{1/2}(\Gamma)\rightarrow H^{1/2}(\Gamma)$,
$\adlo:H^{-1/2}(\Gamma)\rightarrow H^{-1/2}(\Gamma)$, and
$\hyp: H^{1/2}(\Gamma)\rightarrow H^{-1/2}(\Gamma)$.
We note that here holds $\ker(\hyp)=\mathrm{span}\{1\} = \ker(\dlo+\tfrac12)$ and
\begin{align*}
  \norm{v}{H^{1/2}(\Gamma)}^2 \lesssim \dual{\hyp v}{v}_\Gamma \quad\text{for all } v\in H^{1/2}(\Gamma) \quad\text{with }
  \dual{1}v_\Gamma = 0.
\end{align*}
Moreover, assume that $\diam(\Omega)<1$ for $d=2$. This ensures the coercivity estimate
\begin{align*}
  \norm{\psi}{H^{-1/2}(\Gamma)}^2 &\lesssim \dual{\psi}{\slo\psi}_\Gamma \quad\text{for all } \psi\in H^{-1/2}(\Gamma).
\end{align*}
Furthermore, for the exterior Cauchy data of the harmonic function $u^c\in H^1_{\mathrm{loc}}(\Omc)$
(see~\eqref{tpb}) we have the so-called Calder\'on system
\begin{align}\label{BIE}
  \begin{pmatrix}
    \hyp &  \tfrac12+\adlo \\
  \tfrac12-\dlo  & \slo
  \end{pmatrix} 
  \begin{pmatrix} u^c|_\Gamma \\ \partial_{\nn_\Omega}u^c \end{pmatrix} =
  \begin{pmatrix} 0 \\ 0 \end{pmatrix}.
\end{align}
For details and proofs we refer to classical references, e.g.
\cite{Costabel_88_BIO,HsiaoW_08_BIE,McLean_00_SES}.

\subsection{Coupling scheme} \label{sec_schemes}
We tackle the interior problem~\eqref{tpa} by solving the ultra-weak formulation~\eqref{eq:ultraweak}
by the DPG method with optimal test functions
\begin{align}\label{eq:dpg:interior}
  b(\uu,\Theta_\beta\vv) = L_V(\Theta_\beta\vv) \quad\text{for all } \vv\in U.
\end{align}
Note that solutions of this equation are not unique.
To incorporate the exterior problem~\eqref{tpb} we make use of the Calder\'on system~\eqref{BIE} and the
transmission conditions~\eqref{tpc}--\eqref{tpd}.
To that end define the bilinear form $c : (H^{1/2}(\Gamma)\times H^{-1/2}(\Gamma)) \times (H^{1/2}(\Gamma)\times
H^{-1/2}(\Gamma)) \to \R$ by
\begin{align}
\begin{split}
  c( (\wat u,\wat\sigma),(\wat v,\wat\tau) ) &:= \dual{\hyp \wat u}{\wat v}_\Gamma +
  \dual{(\tfrac12+\adlo)\eps^{3/4}\wat\sigma}{\wat v}_\Gamma + \dual{\eps^{3/4}\wat\tau}{(\tfrac12-\dlo)\wat u}_\Gamma + 
  \dual{\eps^{3/4}\wat\tau}{\slo\eps^{3/4}\wat\sigma}_\Gamma \\
  &\qquad\quad + \dual{1}{(\tfrac12-\dlo)\wat u + \slo\eps^{3/4}\wat\sigma}_\Gamma 
  \dual{1}{(\tfrac12-\dlo)\wat v + \slo\eps^{3/4}\wat\tau}_\Gamma.
\end{split}
\end{align}
Then, if $u$ is the solution to the interior problem~\eqref{tpa}, the conditions~\eqref{tpc}--\eqref{tpd} yield
\begin{align}
  c((u|_\Gamma,\eps^{1/4}\partial_{\nn_\Omega}u),\gamma\vv) =  c( (u_0,\eps^{-3/4}\phi_0),\gamma\vv) =:
  L_\BEC(\gamma\vv) \quad\text{for all } \vv\in U
\end{align}
and our coupling scheme reads as follows: Find $\uu\in U$ such that
\begin{align}\label{eq:coupling}
  b(\uu,\Theta_\beta\vv) + \eps^{-1/2} c(\gamma\uu,\gamma\vv) = L_V(\Theta_\beta\vv) + \eps^{-1/2} L_\BEC(\gamma\vv)
  \quad\text{for all } \vv\in U.
\end{align}
The discrete formulation follows immediately by replacing $U$ with some finite dimensional
subspace $U_{hp}\subseteq U$: Find $\uu_{hp}\in U$ such that
\begin{align}\label{eq:coupling:discrete}
  b(\uu_{hp},\Theta_\beta\vv) + \eps^{-1/2} c(\gamma\uu_{hp},\gamma\vv) = L_V(\Theta_\beta\vv) + \eps^{-1/2} L_\BEC(\gamma\vv)
  \quad\text{for all } \vv\in U_{hp}.
\end{align}
The weights $\eps^{-1/2}$ stem from our analysis and are due to the fact that one has to relate the norm of the 
kernel of $B$ corresponding to the interior problem with the norm associated to the exterior problem.

We introduce the following norms on $U$:
\begin{align*}
  \norm{\uu}{\eps,1}^2 &:= \norm{u}{L_2(\Omega)}^2 + \norm{\ssigma}{L_2(\Omega)}^2 + \eps \norm{\rho}{L_2(\Omega)}^2  \\
                           &\qquad+ \eps^{3/2} \norm{\wat u^a}{1/2,\cS}^2 + \eps \norm{\wat u^b}{1/2,\cS}^2 
                           + \eps^{3/2} \norm{\wat\sigma^a}{-1/2,\cS}^2 + \eps^{5/2} \norm{\wat\sigma^b}{-1/2,\cS}^2,
                           \\
  \norm{\uu}{\eps,2}^2 &:= \norm{u}{L_2(\Omega)}^2 + \norm{\ssigma}{L_2(\Omega)}^2 + \eps \norm{\rho}{L_2(\Omega)}^2  \\
  &\qquad+ \norm{\wat u^a}{1/2,\cS}^2 + \eps^{-1/2} \norm{\wat u^b}{1/2,\cS}^2 
                           + \norm{\wat\sigma^a}{-1/2,\cS}^2 + \eps^{1/2} \norm{\wat\sigma^b}{-1/2,\cS}^2,
                           \\
  \enorm{\uu}_j^2 &:= \norm{\uu}{\eps,j}^2 + \eps^{-1/2} \left( \norm{\wat u^a}{H^{1/2}(\Gamma)}^2 +
  \eps^{3/2}\norm{\wat\sigma^a}{H^{-1/2}(\Gamma)}^2 \right) \quad\text{for }j=1,2, \\
  \norm{\uu}E^2 &:= \norm{B\uu}{V'}^2 + \eps^{-1/2} ( \dual{\hyp\wat u^a}{\wat u^a}_\Gamma 
  + \dual{\eps^{3/4}\wat\sigma^a}{\slo\eps^{3/4}\wat\sigma^a}_\Gamma + |\dual{1}{(\tfrac12-\dlo)\wat u^a +
    \slo\eps^{3/4}\wat\sigma^a}_\Gamma|^2).
\end{align*}
The following is our main result.

\begin{theorem} \label{thm:main}
  In the case $d=2$ assume that $\diam(\Omega)<1$.
  There exists $\beta_0>0$ that depends only on $\Omega$ but not on $\eps$ 
  such that for all $\beta\geq\beta_0$ 
  problems~\eqref{tp} and~\eqref{eq:coupling} are equivalent and uniquely solvable.
  More precisely, let $(u,u^c)$ be the solution to \eqref{tp} and define
  \begin{align*}
    \uu:=(u,\ssigma,\rho,\wat u^a,\wat u^b,\wat\sigma^a,\wat\sigma^b) \text{ with }
    \ssigma := \eps^{1/4} \nabla u, \rho:= \div(\ssigma), \wat u^\star := u|_\cS, \wat\sigma^\star :=
    \ssigma\cdot\nn|_\cS
  \end{align*}
  for $\star\in\{a,b\}$.
  (Here, the notation $u|_\cS$ and $\ssigma\cdot\nn|_\cS$ is to be understood in the
  product sense of the spaces $H^{1/2}(\cS)$ and $H^{-1/2}(\cS)$, respectively.)
  Then, $\uu$ solves \eqref{eq:coupling}.
  On the other hand, if $\uu$ solves \eqref{eq:coupling},
  then $(u,u^c)$ with $u^c:=\dlp(\wat u^\star|_\Gamma-u_0) - \slp(\eps^{3/4}\wat\sigma^\star|_\Gamma - \phi_0)$
  solves \eqref{tp}.
  \\
  (i) The bilinear forms are bounded in the sense that
  \begin{subequations}\label{eq:prop:varform}
  \begin{align}\label{eq:prop:varform:bou}
    b(\uu,\Theta_\beta\vv) + \eps^{-1/2} c(\gamma\uu,\gamma\vv) &\lesssim \max\{\beta,1\} \norm{\uu}E\norm{\vv}E
  \end{align}
  for all $\uu,\vv \in U$.
  \\
  (ii) The bilinear form $b(\cdot,\Theta_\beta\cdot) + \eps^{-1/2} c(\gamma\cdot,\gamma\cdot)$
  is coercive: For all $\uu\in U$ holds
  \begin{align}\label{eq:prop:varform:ell}
    \norm{\uu}E^2 \lesssim b(\uu,\Theta_\beta\uu) + \eps^{-1/2} c(\gamma\uu,\gamma\uu).
  \end{align}
  \\
  (iii) There holds 
  \begin{align}\label{eq:prop:varform:equiv}
    \enorm{\uu}_1 \lesssim \norm{\uu}E \lesssim \enorm{\uu}_2 \quad\text{for all }\uu\in U.
  \end{align}
  \end{subequations}
\end{theorem}
A proof of this theorem will be given in Section~\ref{sec:proof:main}.

\begin{remark}
  From the proof of Theorem~\ref{thm:main} below, it is clear that $\beta_0$ depends only
  on equivalence constants involving norms and boundary integral operators on $\Gamma$,
  and on the stability constant of the operator $B : U \to V'$ (see Lemma~\ref{la_kerB}).
  In general $\beta\geq\beta_0\geq 1$ (cf. estimate~\eqref{beta_zero} in the proof).
  The choice $\beta=1$ proved to be sufficient in our numerical simulations.
  In the case of the Laplace transmission problem on an L-shaped domain, cf.~\cite{FuehrerHK_CDB},
  even $\beta>\tfrac14$ suffices.
  In practice we also have to take care of an approximation of the trial-to-test operator $\Theta$.
  Such an approximation is obtained by replacing the space $V$ in \eqref{eq:ttop} by a finite-dimensional subspace
  $V_{hp}$. The resulting approximation of $\Theta$ by a discrete operator $\Theta_h:\;U\to V_{hp}$ can be
  included in the lower bound for $\beta$.
  Assume that $\norm{B\uu_h}{V'}\leq C_F \norm{B\uu_h}{V_{hp}'}$. Then, substituting $\beta_0$ by $C_F\beta_0$
  the assertions of Theorem~\ref{thm:main} (coercivity and boundedness of the combined bilinear form)
  remain valid when replacing $\Theta$ by $\Theta_h$. Note that the constant $C_F$ amounts to
  the bound of an appropriate Fortin operator mapping $V\to V_{hp}$, cf.~\cite{GopalakrishnanQ_14_APD} for
  details in the case of the Poisson equation.
  In the present case of reaction dominated diffusion problems the analysis of the approximation properties
  when replacing $V$ by $V_{hp}$ is still an open issue.
\end{remark}

The following corollary is a direct consequence of Theorem~\ref{thm:main} and the Lax-Milgram lemma.

\begin{cor}\label{cor:bestapprox}
  Let the assumptions of Theorem~\ref{thm:main} hold and let $U_{hp} \subseteq U$ be a finite dimensional subspace.
  Denote by $\uu\in U$ the solution of~\eqref{eq:coupling} and by $\uu_{hp}\in U_{hp}$ the solution
  of~\eqref{eq:coupling:discrete}.
  Then,
  \begin{align}
    \enorm{\uu-\uu_{hp}}_1 \lesssim \max\{\beta,1\} \inf_{\vv_{hp}\in U_{hp}} \norm{\uu-\vv_{hp}}E
    \lesssim \max\{\beta,1\} \inf_{\vv_{hp}\in U_{hp}} \enorm{\uu-\vv_{hp}}_2.
  \end{align}
  \qed
\end{cor}

\begin{remark} \label{rem:eff}
Since the norms $\enorm{\cdot}_1$ and $\enorm{\cdot}_2$ are not uniformly equivalent in $\eps$,
Corollary~\ref{cor:bestapprox} does not prove uniform best approximation of the field variables in
the balanced norm. We do have uniform control of the error (of the field variables) in the balanced
norm by the DPG energy error, and this energy error is being (quasi-) minimized by the DPG approximation.
The lack of uniform equivalence of $\enorm{\cdot}_2$ and the balanced norm
means that we have not proved $\eps$-uniformly \emph{efficient} control of the error in balanced norm
by the energy error. Looking at the case of reaction-dominated diffusion (not the transmission problem)
where the DPG-energy error is exactly minimized (assuming exact implementation),
numerical results indicate that we do have uniform equivalence of the energy error and the error
of the field variables in balanced norm once the boundary layers are resolved, cf.~\cite{HeuerK_RDM}.
Their ratio tends to a number close to $1$ independently of $\eps$ when the mesh is refined.
For the transmission problem a similar behavior is observed (see the numerical results below,
in particular Figure~\ref{fig:ratio}).
In order to \emph{guarantee} uniform best approximation of the field variables in the balanced norm
one can select, e.g., higher-order approximations of the corresponding skeleton terms.
For instance, considering quasi-uniform meshes of width $h$ and assuming sufficient regularity,
an approximation of $\wat u^a$ and $\wat u^b$ by
piecewise quadratic functions (and lowest order approximations of the remaining variables)
would give an additional factor $h$ in the best approximations of $\wat u^a$ and $\wat u^b$.
Then, selecting meshes with $\eps^{-1/4}h=O(1)$ we expect a uniform upper bound in the balanced norm.
To prove such results one needs appropriate approximation properties in skeleton spaces with
balanced trace norms. This is an open subject.
\end{remark}

\section{Technical results and proof of Theorem~\ref{thm:main}} \label{sec:proof}
In the recent work~\cite{FuehrerHK_CDB} we have learned that we have to analyze the kernel
of $B$ and relate it to the exterior problem. The candidates of the kernel will be solutions of the
problem
\begin{align}\label{eq:ext}
\begin{split}
  -\eps \Delta \widetilde u + \widetilde u &= 0 \quad\text{in } \Omega, \\
   \widetilde u|_\Gamma &= \wat u \quad\text{on } \Gamma.
\end{split}
\end{align}
We refer to these solutions as \textit{harmonic} extension of $\wat u$ and define the operator 
$\ext : H^{1/2}(\Gamma) \to U$ by
\begin{align*}
  \ext\wat u &:= \uu = (\widetilde u,\eps^{1/4}\nabla \widetilde u,\eps^{1/4}\Delta\widetilde u,
  \widetilde u|_\cS,\widetilde u|_\cS, \eps^{1/4}(\nabla\widetilde u\cdot\nn_T|_{\partial T})_{T\in\TT},
  \eps^{1/4}(\nabla\widetilde u\cdot\nn_T|_{\partial T})_{T\in\TT}),
\end{align*}
where $\widetilde u\in H^1(\Omega)$ denotes the unique solution of~\eqref{eq:ext}.

We can state the following properties of $\ext$.
\begin{lemma} \label{la_ext}
  The operator $\ext:\; (H^{1/2}(\Gamma),\norm{\cdot}{H^{1/2}(\Gamma)})\to (U,\norm{\cdot}{\eps,1})$ is linear, a right-inverse of $\gamma_0$, and bounded in the sense of
\begin{align*}
  \norm{\ext \wat u}{\eps,1} \lesssim \eps^{-1/4} \norm{\wat u}{H^{1/2}(\Gamma)}.
\end{align*}
\end{lemma}
\begin{proof}
By construction we have $\gamma_0\ext\wat u=\wat u$ for any $\wat u\in H^{1/2}(\Gamma)$,
showing that $\ext$ is a right-inverse.
Let $\ext \wat u := \uu = (u,\ssigma,\rho,\wat u^a,\wat u^b,\wat\sigma^a,\wat\sigma^b)$ denote the components as
defined above.
The continuity of $\ext$ follows from the well-posedness of the Dirichlet problem: The weak formulation of the
homogeneous problem and the Cauchy-Schwarz inequality gives us
\begin{align*}
  \eps\norm{\nabla u}{L_2(\Omega)}^2 +\norm{u}{L_2(\Omega)}^2 = \eps\dual{\partial_{\nn_\Gamma} u}{\widehat u}_\Gamma 
  &= \eps\ip{\nabla u}{\nabla v} +\ip{u}{v} \\
  &\leq (\eps\norm{\nabla u}{L_2(\Omega)}^2 +\norm{u}{L_2(\Omega)}^2)^{1/2}
  (\eps\norm{\nabla v}{L_2(\Omega)}^2 +\norm{v}{L_2(\Omega)}^2)^{1/2}
\end{align*}
for any extension $v\in H^1(\Omega)$ with $v|_\Gamma = \widehat u$.
Dividing by the term of the right-hand side which involves $u$ and taking the infimum over all extensions $v\in
H^1(\Omega)$ we obtain
\begin{align*}
  (\eps\norm{\nabla u}{L_2(\Omega)}^2 +\norm{u}{L_2(\Omega)}^2)^{1/2} \leq \norm{\wat u}{1/2,\Gamma}
  \leq \norm{\wat u}{H^{1/2}(\Gamma)}.
\end{align*}
Note that $\norm{\wat u^\star}{1/2,\cS}^2\leq \norm{u}{L_2(\Omega)}^2 + \norm{\ssigma}{L_2(\Omega)}^2$ and
$\norm{\wat \sigma^\star}{-1/2,\cS}^2\leq \norm{\ssigma}{L_2(\Omega)}^2 + \eps\norm{\rho}{L_2(\Omega)}^2$.
By definition of $\norm{\cdot}{\eps,1}$ we get
\begin{align*}
  \norm{\ext\wat u}{\eps,1}^2 &:= \norm{u}{L_2(\Omega)}^2 + \norm{\ssigma}{L_2(\Omega)}^2 + \eps \norm{\rho}{L_2(\Omega)}^2  \\
                           &\qquad+ \eps^{3/2} \norm{\wat u^a}{1/2,\cS}^2 + \eps \norm{\wat u^b}{1/2,\cS}^2 
                           + \eps^{3/2} \norm{\wat\sigma^a}{-1/2,\cS}^2 + \eps^{5/2} \norm{\wat\sigma^b}{-1/2,\cS}^2 \\
                      & \lesssim \norm{u}{L_2(\Omega)}^2 + \norm{\ssigma}{L_2(\Omega)}^2 + \eps \norm{\rho}{L_2(\Omega)}^2. 
\end{align*}
Also by definition it follows
\begin{align*}
  \norm{\ext\wat u}{\eps,1}^2 \geq \norm{u}{L_2(\Omega)}^2 + \norm{\ssigma}{L_2(\Omega)}^2 + \eps
  \norm{\rho}{L_2(\Omega)}^2.
\end{align*}
Therefore, using $\ssigma = \eps^{1/4}\nabla u$, $\rho = \eps^{1/4}\Delta u =
\eps^{-3/4} u$, we have
\begin{align*}
  \norm{\ext\wat u}{\eps,1}^2 \simeq \norm{u}{L_2(\Omega)}^2 + \eps^{1/2}\norm{\nabla u}{L_2(\Omega)}^2 +
  \eps^{-1/2} \norm{u}{L_2(\Omega)}^2 \simeq \eps^{-1/2}\left(\norm{u}{L_2(\Omega)}^2 + \eps\norm{\nabla u}{L_2(\Omega)}^2\right).
\end{align*}
Thus, $\norm{\ext\wat u}{\eps,1}\lesssim \eps^{-1/4} \norm{\wat u}{H^{1/2}(\Gamma)}$.
\end{proof}

\begin{lemma} \label{la_kerB}
There holds $\ker B = \ext H^{1/2}(\Gamma)$ and, in particular, 
\begin{align*}
  \enorm{\uu-\ext\gamma_0\uu}_1 \lesssim \norm{B\uu}{V'} \quad\text{for all } \uu\in U.
\end{align*}
The involved constant depends only on $\Omega$ but not on $\eps$.
\end{lemma}

\begin{proof}
We note that $\ext H^{1/2}(\Gamma)\subseteq \ker B$ follows 
by construction of the ultra-weak formulation~\eqref{eq:ultraweak}, i.e., integration by parts.

To show the other inclusion, note that $\gamma_0(\uu-\ext\gamma_0\uu) = 0$, hence, $\uu-\ext\gamma_0\uu \in U_0$.
Recall from~\cite[Theorem~1]{HeuerK_RDM} that
\begin{align*}
  \norm{\uu_0}{\eps,1} \lesssim \norm{B\uu_0}{V'} \quad\text{for all } \uu_0\in U_0.
\end{align*}
Since $\ext\gamma_0\uu\in \ker B$ we have
\begin{align*}
  \norm{\uu-\ext\gamma_0\uu}{\eps,1} \lesssim \norm{B(\uu-\ext\gamma_0\uu)}{V'} = \norm{B\uu}{V'} \quad\text{for all }
  \uu\in U.
\end{align*}
Choosing $\uu\in \ker B$ shows $\uu=\ext\gamma_0\uu \in \ext H^{1/2}(\Gamma)$.
\end{proof}

The following norm equivalence follows by a standard compactness argument.
Similar results for classical coupling methods can be found in~\cite{AuradaFFKMP_13_CFB}.

\begin{lemma}\label{la_normEquiv}
  In the case $d=2$ assume that $\diam(\Omega)<1$. There holds 
  \begin{align*}
    \norm{u}{H^{1/2}(\Gamma)}^2 + \norm{\phi}{H^{-1/2}(\Gamma)}^2 \simeq \dual{\hyp u}{u}_\Gamma +
    \dual\phi{\slo\phi}_\Gamma + |\dual{1}{(\tfrac12-\dlo)u + \slo\phi}_\Gamma|^2
  \end{align*}
  for all $(u,\phi)\in H^{1/2}(\Gamma)\times H^{-1/2}(\Gamma)$.
  In particular,
  \begin{align*}
    \norm{u}{H^{1/2}(\Gamma)}^2 + \norm{\eps^{3/4}\phi}{H^{-1/2}(\Gamma)}^2 \simeq \dual{\hyp u}{u}_\Gamma +
    \dual{\eps^{3/4}\phi}{\slo\eps^{3/4}\phi}_\Gamma + |\dual{1}{(\tfrac12-\dlo)u + \slo\eps^{3/4}\phi}_\Gamma|^2
  \end{align*}
  and the involved constants depend only on $\Gamma$.
  \qed
\end{lemma}

We shall make use of the following bound for the $H^{-1/2}(\Gamma)$ norm.
\begin{lemma}\label{lem:boundHm12}
  There holds
  \begin{align*}
    \eps^{1/2} \norm{\gamma_\nn^a\uu}{H^{-1/2}(\Gamma)} \lesssim \norm{B\uu}{V'} + \norm{\uu}{\eps,1} 
    \quad\text{for all } \uu\in U.
  \end{align*}
\end{lemma}
\begin{proof}
  For given $\widetilde\mu\in H^{1/2}(\Gamma)$ let $\mu\in H^1(\Omega)$ be its extension with
  $\mu|_\Gamma = \widetilde\mu$ and $\norm{\mu}{H^1(\Omega)} = \norm{\widetilde\mu}{H^{1/2}(\Gamma)}$.
  Since $\mu\in H^1(\Omega)$ there holds
  \begin{align*}
    \dual{\mu}{\ttau\cdot\nn}_{\cS} = 0 \quad\text{for all } \ttau\in \HH(\div,\Omega) \text{ with }
    \ttau\cdot\nn_\Omega|_\Gamma = 0.
  \end{align*}
  Let $\uu = (u,\ssigma,\rho,\wat u^a,\wat u^b, \wat \sigma^a,\wat\sigma^b) \in U$. Then,
  \begin{align*}
    \dual{\wat\sigma^a}{\widetilde\mu}_\Gamma &= \dual{\wat\sigma^a}{\mu}_\cS = -\dual{\wat\sigma^a}{-\mu}_\cS
    + \ip{\ssigma}{\pwnabla(-\mu)} + \ip{\rho}{-\mu} 
     + \ip{\ssigma}{\pwnabla\mu} + \ip{\rho}{\mu} \\
     &\leq b(\uu,(-\mu,0,0)) + (\norm{\ssigma}{L_2(\Omega)}^2+\norm\rho{L_2(\Omega)}^2)^{1/2}\norm{\mu}{H^1(\Omega)}
  \end{align*}
  Moreover, note that with the definition of the test norm $\norm\cdot{V}$ we get
  \begin{align*}
    \norm{\mu}{H^1(\Omega)}^2 \geq \eps \norm{(\mu,0,0)}{V}^2.
  \end{align*}
  This leads
  \begin{align*}
    \norm{\gamma_\nn^a\uu}{H^{-1/2}(\Gamma)} & \simeq \sup_{\widetilde\mu\in H^{1/2}(\Gamma)}
    \frac{\dual{\gamma_\nn^a\uu}{\widetilde\mu}_\Gamma}{\norm{\widetilde\mu}{H^{1/2}(\Gamma)}}
    = \sup_{\mu\in H^{1}(\Omega)}
    \frac{\dual{\gamma_\nn^a\uu}{\mu|_\Gamma}_\Gamma}{\norm{\mu}{H^1(\Omega)}} 
    \\ & \leq \eps^{-1/2} \left(\sup_{\mu\in H^{1}(\Omega)}
    \frac{b(\uu,(-\mu,0,0))}{\norm{(-\mu,0,0)}{V}}\right) + (\norm{\ssigma}{L_2(\Omega)}^2 + \norm{\rho}{L_2(\Omega)}^2)^{1/2}.
  \end{align*}
  Extending the supremum over all functions in $V$ and using the definition of $\norm{\uu}{\eps,1}$, this finishes the proof.
\end{proof}

\subsection{Proof of Theorem~\ref{thm:main}}\label{sec:proof:main}
Integration by parts and the Calder\'on system \eqref{BIE} show that the unique solution $\uu$ of \eqref{tp} 
as defined in the theorem also solves problem~\eqref{eq:coupling}.
It remains to prove that~\eqref{eq:coupling} is uniquely solvable. 
To that end, we prove boundedness (i) and strong coercivity (ii) below. Then, by the Lax-Milgram lemma we have a 
unique solution of~\eqref{eq:coupling}, which also holds in the discrete case~\eqref{eq:coupling:discrete}.

\noindent
\emph{Proof of (i): }
Boundedness in the norm $\norm\cdot{E}$ is straightforward: 
Note that $b(\uu,\Theta_\beta\vv)$ defines a symmetric, positive semidefinite bilinear form. Therefore, we can apply
the Cauchy-Schwarz inequality 
\begin{align*}
  |b(\uu,\Theta_\beta\vv)|\leq b(\uu,\Theta_\beta\uu)^{1/2} b(\vv,\Theta_\beta\vv)^{1/2} =
  \beta \norm{B\uu}{V'}\norm{B\vv}{V'}.
\end{align*}
Then, the mapping properties of the boundary integral operators from Section~\ref{sec_BIO} yield
\begin{align*}
  |c(\gamma\uu,\gamma\vv)| \lesssim (\norm{\gamma_0\uu}{H^{1/2}(\Gamma)}^2 
  + \norm{\eps^{3/4}\gamma_\nn^a\uu}{H^{-1/2}(\Gamma)}^2)^{1/2}  (\norm{\gamma_0\vv}{H^{1/2}(\Gamma)}^2 
  + \norm{\eps^{3/4}\gamma_\nn^a\vv}{H^{-1/2}(\Gamma)}^2)^{1/2}.
\end{align*}
Together with Lemma~\ref{la_normEquiv} this gives
\begin{align*}
  |b(\uu,\Theta_\beta\vv)| + \eps^{-1/2}|c(\gamma\uu,\gamma\vv)| \lesssim \max\{\beta,1\} \norm{\uu}E \norm{\vv}E.
\end{align*}

\noindent
\emph{Proof of (ii): } 
Let $\uu=(u,\ssigma,\rho,\wat u^a,\wat u^b,\wat\sigma^a,\wat\sigma^b)\in U$. 
The definition of the norm $\norm\cdot{E}$ and the identity
\begin{align*}
  \dual{\eps^{3/4}\wat\sigma^a}{\wat u^a}_\Gamma = \dual{(\tfrac12+\adlo)\eps^{3/4}\wat\sigma^a}{\wat u^a}_\Gamma
  + \dual{\eps^{3/4}\wat\sigma^a}{(\tfrac12-\dlo)\wat u^a}_\Gamma
\end{align*}
imply
\begin{align*}
  \norm{\uu}E^2 &= \norm{B\uu}{V'}^2 + \eps^{-1/2} ( \dual{\hyp\wat u^a}{\wat u^a}_\Gamma 
  + \dual{\eps^{3/4}\wat\sigma^a}{\slo\eps^{3/4}\wat\sigma^a}_\Gamma + |\dual{1}{(\tfrac12-\dlo)\wat u^a +
    \slo\eps^{3/4}\wat\sigma^a}_\Gamma|^2)
  \\
  &= \norm{B\uu}{V'}^2 + \eps^{-1/2} ( \dual{\hyp\wat u^a}{\wat u^a}_\Gamma 
    + \dual{\eps^{3/4}\wat\sigma^a}{\slo\eps^{3/4}\wat\sigma^a}_\Gamma + |\dual{1}{(\tfrac12-\dlo)\wat u^a +
      \slo\eps^{3/4}\wat\sigma^a}_\Gamma|^2) \\
    &\qquad\qquad + \eps^{-1/2} \dual{(\tfrac12+\adlo)\eps^{3/4}\wat\sigma^a}{\wat u^a}_\Gamma
    + \eps^{-1/2} \dual{\eps^{3/4}\wat\sigma^a}{(\tfrac12-\dlo)\wat u^a}_\Gamma  - \eps^{-1/2} \dual{\eps^{3/4}\wat\sigma^a}{\wat u^a}_\Gamma
  \\
  &=\norm{B\uu}{V'}^2 + \eps^{-1/2}c(\gamma\uu,\gamma\uu) 
  - \eps^{1/4} \dual{\wat\sigma^a}{\wat u^a}_\Gamma \\
  &= \norm{B\uu}{V'}^2 + \eps^{-1/2}c(\gamma\uu,\gamma\uu) + \dual{\eps^{1/2}(\gamma_\nn^a\ext\gamma_0\uu-\wat\sigma^a)}{\eps^{-1/4}\wat u^a}_\Gamma
  - \eps^{1/4} \dual{\gamma_\nn^a\ext\gamma_0\uu}{\wat u^a}_\Gamma \\
  &\leq \norm{B\uu}{V'}^2 + \eps^{-1/2}c(\gamma\uu,\gamma\uu) + \dual{\eps^{1/2}(\gamma_\nn^a\ext\gamma_0\uu-\wat\sigma^a)}{\eps^{-1/4}\wat u^a}_\Gamma \\
  &\leq \norm{B\uu}{V'}^2 + \eps^{-1/2}c(\gamma\uu,\gamma\uu) + \eps^{1/2} \norm{\gamma_\nn^a\ext\gamma_0\uu-\wat\sigma^a}{H^{-1/2}(\Gamma)}
  \norm{\eps^{-1/4}\wat u^a}{H^{1/2}(\Gamma)}.
\end{align*}
Here, we also used that $\dual{\gamma_\nn^a\ext\gamma_0\uu}{\wat u^a}_\Gamma\geq 0$.
We apply Lemma~\ref{lem:boundHm12}, Lemma~\ref{la_kerB} and Young's inequality with parameter $\delta>0$ 
to estimate the last term on the right-hand side by
\begin{align*}
  \eps^{1/2}\norm{\gamma_\nn^a\ext\gamma_0\uu-\wat\sigma^a}{H^{-1/2}(\Gamma)} 
  \norm{\eps^{-1/4} \wat u^a}{H^{1/2}(\Gamma)}
  &\leq C \norm{B\uu}{V'}\norm{\eps^{-1/4} \wat u^a}{H^{1/2}(\Gamma)} \\ 
  &\leq C^2\tfrac{\delta^{-1}}2  \norm{B\uu}{V'}^2 + \tfrac{\delta}2 \norm{\eps^{-1/4} \wat u^a}{H^{1/2}(\Gamma)}^2 \\
  &\leq C^2\tfrac{\delta^{-1}}2  \norm{B\uu}{V'}^2 + \tfrac{\delta}2 \norm{\uu}E^2.
\end{align*}
Here, $C>0$ is a constant which is independent of $\eps$, $\delta$, $\uu$ and $\TT$.
Putting everything together we get
\begin{align} \label{beta_zero}
  \norm{\uu}E^2 \le \left(1+C^2\tfrac{\delta^{-1}}2\right) \norm{B\uu}{V'}^2 + \eps^{-1/2}c(\gamma\uu,\gamma\uu) 
  + \tfrac{\delta}2 \norm{\uu}E^2.
\end{align}
Subtracting the last term for a sufficiently small $\delta>0$ this proves the existence of $\beta_0>0$
such that, for all $\beta\geq\beta_0$, there holds
$\norm{\uu}E^2 \lesssim \beta \norm{B\uu}{V'}^2 + \eps^{-1/2} c(\uu,\uu) 
= b(\uu,\Theta_\beta\uu) + \eps^{-1/2} c(\uu,\uu)$.

\noindent
\emph{Proof of (iii): } 
We start with the lower bound in~\eqref{eq:prop:varform:equiv}.
By the triangle inequality, Lemma~\ref{la_kerB}, Lemma~\ref{la_ext}, and Lemma~\ref{la_normEquiv} we obtain
\begin{align}
  \norm{\uu}{\eps,1}^2 \lesssim \norm{\uu-\ext\gamma_0\uu}{\eps,1}^2 +  
  \norm{\ext\gamma_0\uu}{\eps,1}^2 \lesssim \norm{B\uu}{V'}^2 + \eps^{-1/2} \norm{\gamma_0\uu}{H^{1/2}(\Gamma)}^2
  \lesssim \norm{\uu}E^2,
\end{align}
and it remains to show the upper bound in~\eqref{eq:prop:varform:equiv}.
By~\cite[Lemma~3]{HeuerK_RDM} we have
\begin{align*}
  |b(\uu,\Theta\vv)| \lesssim \norm{\uu}{\eps,2}\norm{\Theta\vv}{V}.
\end{align*}
The definition~\eqref{eq:ttop} of $\Theta$ shows
\begin{align*}
  \norm{\Theta\vv}{V}^2 = \dual{\Theta\vv}{\Theta\vv}_V = b(\vv,\Theta\vv) \lesssim
  \norm{\vv}{\eps,2}\norm{\Theta\vv}V,
\end{align*}
and therefore boundedness $b(\uu,\Theta\vv)\lesssim \norm{\uu}{\eps,2} \norm{\vv}{\eps,2}$ for all $\uu,\vv\in U$.
Then, as in the proof of~{(i)} above, boundedness of the involved integral operators show
the upper bound.
This finishes the proof of Theorem~\ref{thm:main}.
\qed

\section{Numerical experiments}\label{sec:num}
In this section we present different numerical experiments of our coupling method in two dimensions.
In order to show numerically the reliability and robustness of our method for the interior singularly perturbed
problem, the first example (Section~\ref{sec:num:smooth}) treats a problem with known solution in the interior and
exterior.
Also, for the second problem (Section~\ref{sec:num:singular}) we choose a known solution that exhibits a
singularity at a boundary vertex.
Lastly, similar to~\cite{HeuerK_RDM}, we consider a problem where the source term $f$ has support within $\Omega$ and
vanishes on the boundary. In particular, we do not only expect layers at the boundary but also in the 
interior (away from the boundary), which are not aligned with the meshes.

Throughout, we consider regular triangulations $\TT$ with compact triangles. 
As mesh refinement strategy we use the \emph{newest-vertex bisection} (NVB) that preserves shape regularity in the
following sense: Given a sequence $\TT_0,\TT_1,\dots$ of triangulations, where $\TT_{\ell}$ is a refinement of $\TT_{\ell-1}$, it
holds
\begin{align*}
  \sup_{T\in\TT_{\ell}} \frac{\diam(T)^2}{|T|} \leq C \sup_{T\in\TT_{0}} \frac{\diam(T)^2}{|T|} \quad\text{for all }
  \ell\in\N_0, 
\end{align*}
where $C$ does not depend on $\ell$.
Given a triangulation $\TT$, we define $P^p(\TT)$ as the space of all $\TT$-piecewise polynomials of order
less than or equal to $p\in\N_0$. The choice for the basis of $P^p(\TT)$ is based on Lobatto shape functions as presented
in~\cite[Section~2.2.2--2.2.3]{SolinSD_04_HOF}. 
In the same manner we define the $\cS$-piecewise polynomial space $P^p(\cS)$. Moreover, set
$S^p(\cS) := P^p(\cS)\cap C(\cS)$. 
We stress that the skeleton $\cS$ restricted to $\Gamma$ induces a triangulation of the boundary. We denote by
$P^p(\cS|_\Gamma)$ the space of $\cS|_\Gamma$-piecewise polynomials of degree less than or equal to $p\in\N_0$ and
$S^p(\cS|_\Gamma) := P^p(\cS|_\Gamma)\cap C(\Gamma)$.
For the lowest-order cases we have $P^0(\cS)|_\Gamma = P^0(\cS|_\Gamma)$ and $S^1(\cS)|_\Gamma = S^1(\cS|_\Gamma)$.
The continuous trial space $U$ is replaced by
\begin{align*}
  U_{hp} := \left(P^0(\TT) \times [P^0(\TT)]^2 \times P^0(\TT) \times S^1(\cS) \times S^1(\cS) \times P^0(\cS) \times
  P^0(\cS)\right) \cap U.
\end{align*}
Moreover, we replace the trial-to-test operator $\Theta_\beta$ by its discrete analogue $\Theta_{{hp},\beta}$
defined by the relation
\begin{align*}
  \dual{\Theta_{{hp},\beta}\uu_{hp}}{\vv_{hp}}_V = \beta b(\uu_{hp},\vv_{hp}) 
  \quad \text{for all } \uu_{hp}\in U_{hp}, \vv_{hp} \in V_{hp},
\end{align*}
where 
\begin{align*}
  V_{hp} := P^2(\TT) \times [P^2(\TT)]^2 \times P^4(\TT).
\end{align*}
That means, we replace $V$ by the discrete subspace $V_{hp}$. In the literature this is often called \emph{practical DPG}.
For the first two components the choice of polynomial degree is based on~\cite{GopalakrishnanQ_14_APD}, whereas for the
third component we choose a polynomial degree of four, since we also test with the Laplacian of the test functions.
We choose $\beta=1$ for all our experiments.

To steer the adaptive mesh-refinement, we use the local indicators
\begin{align*}
  \est_\Omega(T)^2 &:= \norm{R_{hp}^{-1}(L_V-B\uu_{hp})}{V|_T}^2, \\
  \est_\Gamma(E)^2 &:= \eps^{-1/2} \left( \norm{h^{1/2}( \hyp(u_0-\gamma_0\uu_{hp}) 
  + (\tfrac12+\dlo')(\phi_0-\eps^{3/4}\gamma_\nn\uu_{hp})}{L_2(E)}^2  \right.\\ 
  &\qquad\qquad \left. \norm{h^{1/2} \nabla_\Gamma \big( (\tfrac12-\dlo)(u_0-\gamma_0\uu_{hp}) +
  \slo(\phi_0-\eps^{3/4}\gamma_\nn\uu_{hp})\big) }{L_2(E)}^2  \right),
\end{align*}
for $T\in\TT$, $E\in\cS|_\Gamma$. Here, $R_{hp} : V_{hp} \to V_{hp}'$ denotes the Riesz isomorphism,
$h$ denotes the mesh-size function on the boundary and $\nabla_\Gamma(\cdot)$ the arc-length derivative.
We define the overall estimators by
\begin{align*}
  \est_\Omega^2 := \sum_{T\in\TT} \est_\Omega(T)^2, \quad \est_\Gamma^2 := \sum_{E\in\cS|_\Gamma} \est_\Gamma(E)^2.
\end{align*}
Note that $\est_\Omega$ is the typical estimator for DPG problems and can be easily computed. It
measures the error in the interior domain, whereas $\est_\Gamma$ is an $h$-weighted residual-based
error estimator that measures --- at least heuristically --- the error at the boundary. 
We mark a minimal set of elements $\mathcal{M}_\ell^\Omega\times\mathcal{M}_\ell^\Gamma \subseteq \TT_\ell\times
\cS_\ell|_\Gamma$ using the criterion
\begin{align}\label{eq:markcrit}
  \theta(\est_\Omega^2 + \est_\Gamma^2) \leq \sum_{T\in\mathcal{M}_\ell^\Omega} \est_\Omega(T)^2 +
  \sum_{E\in\mathcal{M}_\ell^\Gamma} \est_\Gamma(E)^2
\end{align}
with marking parameter $\theta\in(0,1)$.
For the examples where the exact solution $(u,\uc)$ is known
(Section~\ref{sec:num:smooth}--\ref{sec:num:singular}),
we compute the following error quantities 
(with $\ssigma = \eps^{1/4}\nabla u$, $\rho = \eps^{1/4}\Delta u$):
\begin{align*}
  \err_\Omega^2 &:= \norm{u-u_{hp}}{L_2(\Omega)}^2 + \norm{\ssigma-\ssigma_{hp}}{L_2(\Omega)}^2 + \eps
  \norm{\rho-\rho_{hp}}{L_2(\Omega)}^2, \\
  \err_\Gamma^2 &:= \eps^{-1/2} \norm{h^{1/2}\nabla_\Gamma(\gamma_0\uu_{hp}-u_0-\uc|_\Gamma)}{L_2(\Gamma)}^2 +
  \eps^{-1/2}\norm{h^{1/2}(\eps^{3/4}\gamma_\nn\uu_{hp}-\phi_0-\partial_{\nn_\Omega}\uc)}{L_2(\Gamma)}^2.
\end{align*}

All experiments are programmed and conducted in MATLAB, where for the assembling of the discrete boundary integral
operators we use the MATLAB library HILBERT~\cite{AuradaEFFFGKMP_HMI}.

\subsection{Smooth solution}\label{sec:num:smooth}
\begin{figure}[htb]
  \includegraphics[width=0.49\textwidth]{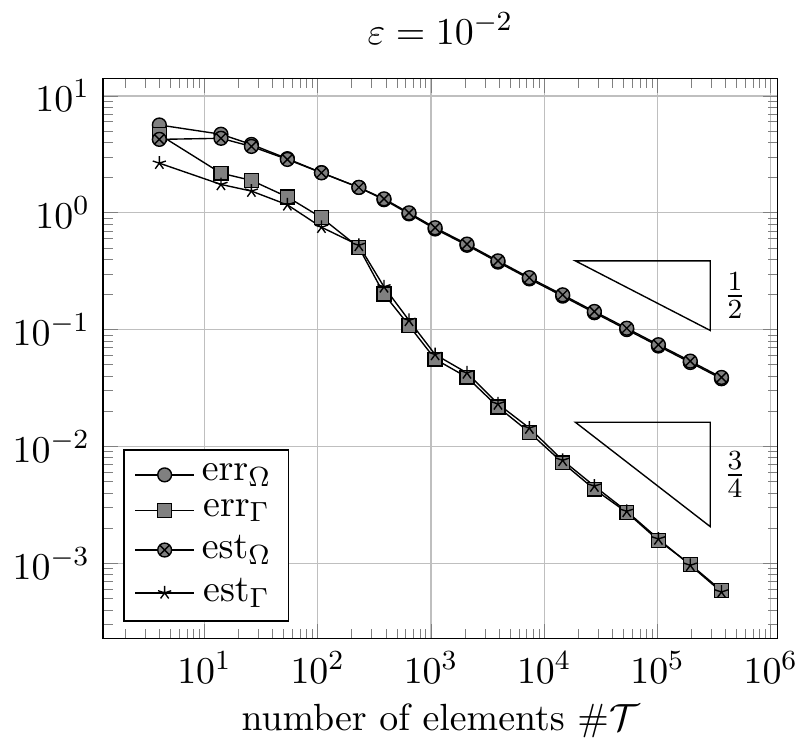}
  \includegraphics[width=0.49\textwidth,height=0.333\textheight]{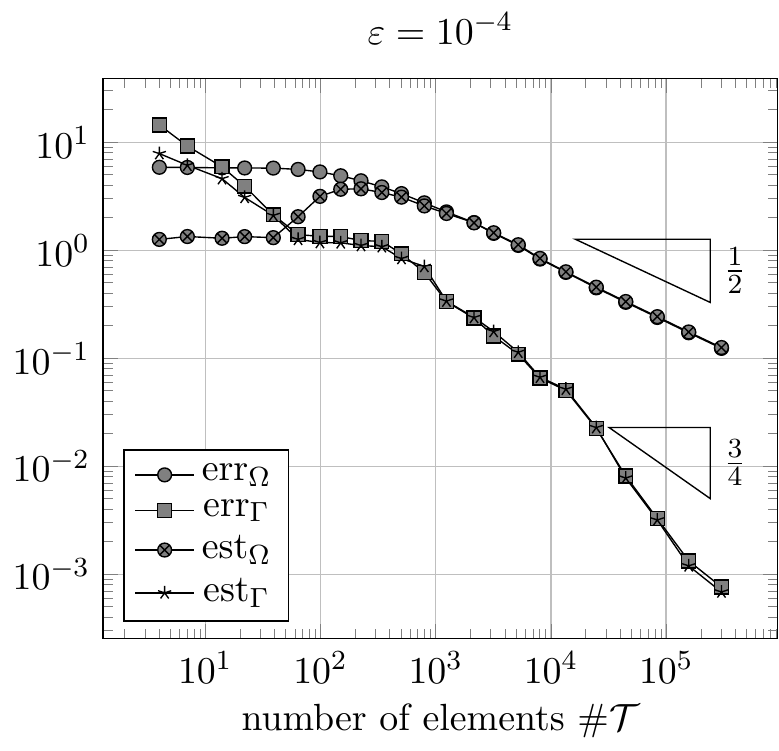}
  \includegraphics[width=0.49\textwidth]{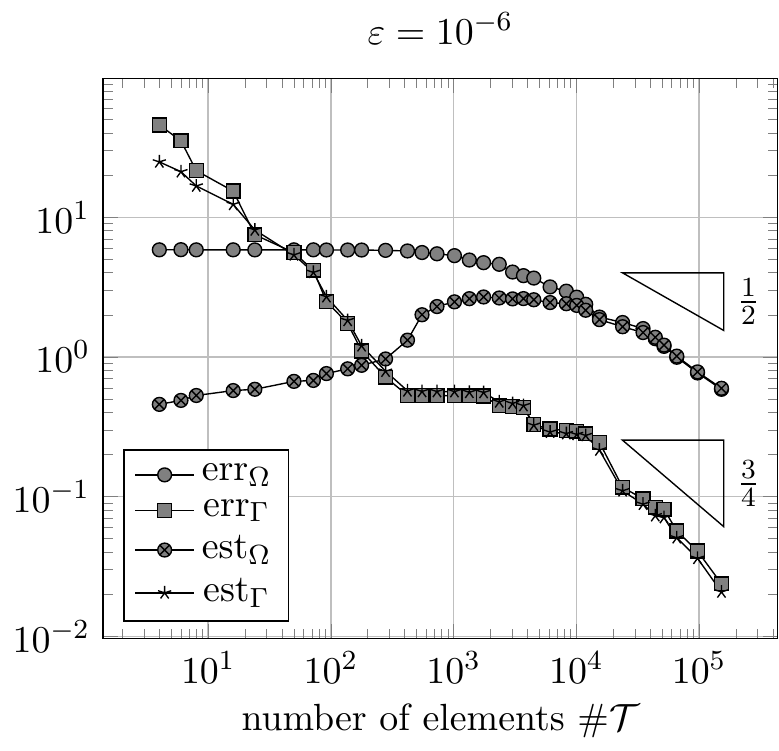}
  \includegraphics[width=0.49\textwidth]{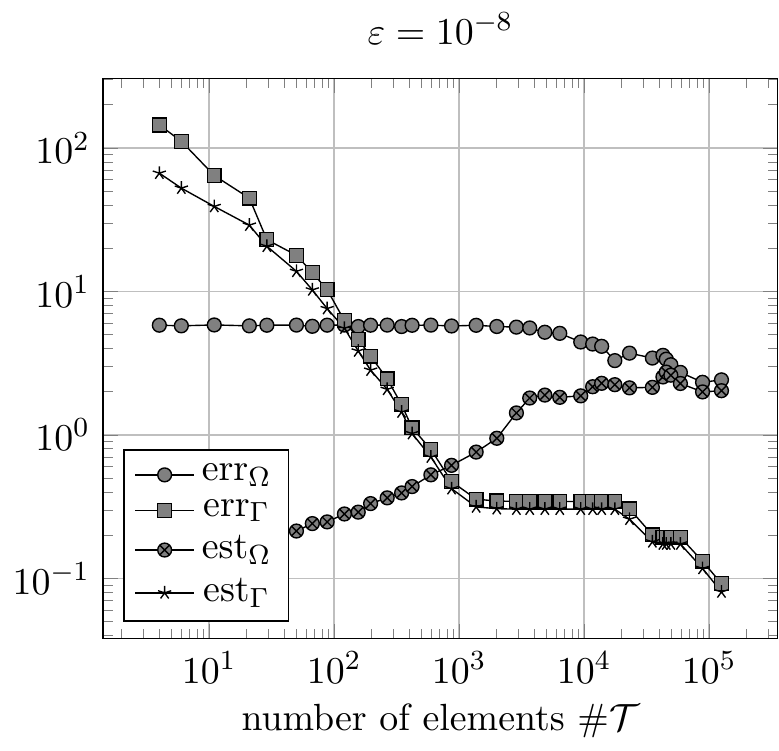}
  \caption{Error and estimator quantities for different values of $\eps$ for the example with smooth solution
    (Section~\ref{sec:num:smooth}).}
  \label{fig:smooth}
\end{figure}
Let $\Omega = (0,\tfrac12)\times (0,\tfrac12)$. We choose the exact solution
\begin{subequations}
\begin{align}\label{eq:smoothSol:int}
  u(x,y) &:= 8(x^3(1+4y^2) + \sin(4\pi x^2) \\ 
  &\qquad + 2 \cos(\pi y) (x+y) \left( e^{-4x/\sqrt{\eps}} + e^{-2(1-x)/\sqrt{\eps}} + e^{-6y/\sqrt{\eps}} +
  e^{-3(1-2y)/\sqrt{\eps}} \right), \nonumber \\ 
  \uc(x,y) &:= \frac1{10} \frac{x+y-\alpha_1-\alpha_2}{(x-\alpha_1)^2 + (y-\alpha_2)^2}, \label{eq:smoothSol:ext}
\end{align}
\end{subequations}
and compute the data $f$, $u_0$, $\phi_0$. 
The solution $u(x,y)$ is similar as in~\cite{LinS_12_BFE,HeuerK_RDM}. Note that we have scaled $\Omega$ such that
$\diam(\Omega)<1$.
We set $\alpha_1 = 0.25 = \alpha_2$.
Then, $\uc$ satisfies $\Delta\uc = 0$ in the exterior domain.
\begin{figure}[htb]
 \includegraphics[width=0.49\textwidth]{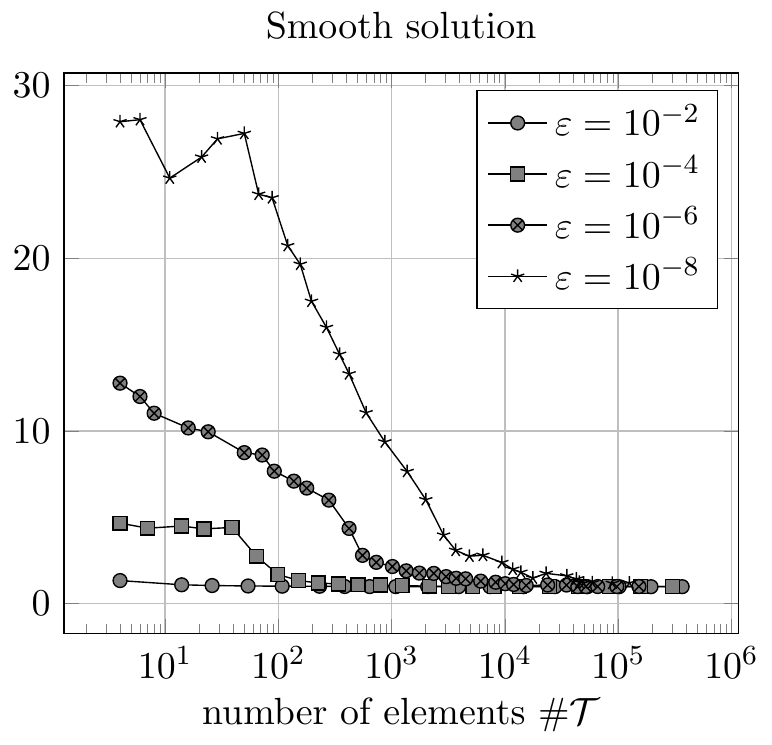}
 \includegraphics[width=0.49\textwidth]{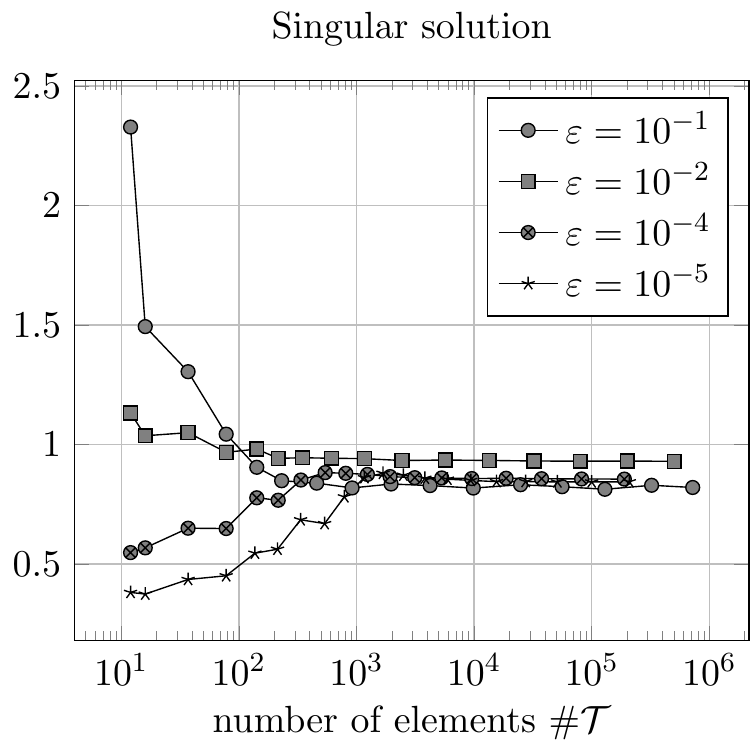}
  \caption{Ratio $\err_\Omega/\est_\Omega$ of the volume error and volume estimator for the example with smooth solution
    (left, Section~\ref{sec:num:smooth}) resp. singular solution (right, Section~\ref{sec:num:singular}).}
  \label{fig:ratio}
\end{figure}

For this experiment we use $\theta=\tfrac12$ in~\eqref{eq:markcrit}. Figure~\ref{fig:smooth} shows the quantities
$\err_\Omega$, $\err_\Gamma$, $\est_\Omega$, and $\est_\Gamma$ for $\eps=10^{-j}$ with $j\in\{2,4,6,8\}$.
We start our computations with a coarse initial mesh $\TT_0$ containing only four elements and observe that as soon as
the boundary layers are resolved our method leads to optimal convergence rate with respect to the number of volume
elements, i.e., $(\#\TT_\ell)^{-1/2}$.
Additionally, we observe higher convergence rates for the boundary terms, that is $(\#\TT_\ell)^{-3/4}$.
For small $\eps$, the boundary estimator $\est_\Gamma$ dominates the overall estimator
$(\est_\Omega^2+\est_\Gamma^2)^{1/2}$, hence, basically boundary elements are refined in the beginning of our adaptive
loop.
This is due to the scaling factor $\eps^{-1/4}$ in the term $\est_\Gamma$.
Nevertheless, after a few steps the volume error dominates. For $\eps=10^{-8}$ this happens already when reaching approximately 1000 volume elements, see Figure~\ref{fig:smooth}.
The left plot in Figure~\ref{fig:ratio} shows the ratio $\err_\Omega/\est_\Omega$ of the volume error and volume
estimator. We observe that this ratio is close to one when the boundary layers are resolved.
In the preasymptotic range a loss of robustness is observed which is due to the fact that the optimal test function 
for $u_{hp}$ can not be resolved accurately enough in the discrete test space $V_{hp}$ on very coarse triangulations.
Nevertheless, the correct location of the layers is found by the proposed algorithm and refinement is done
towards the layers.

\begin{figure}[htb]
  \includegraphics[width=0.49\textwidth]{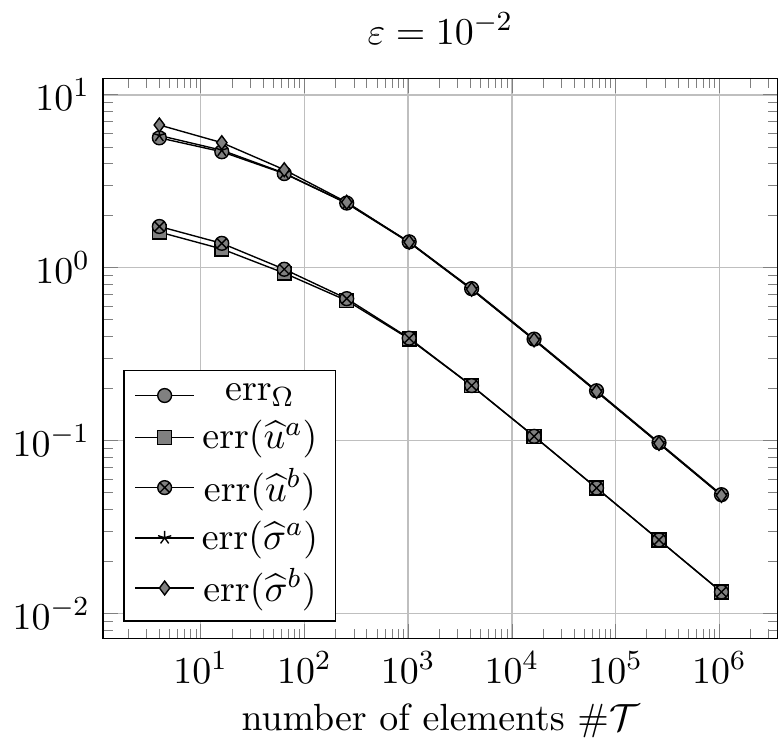}
  \includegraphics[width=0.49\textwidth]{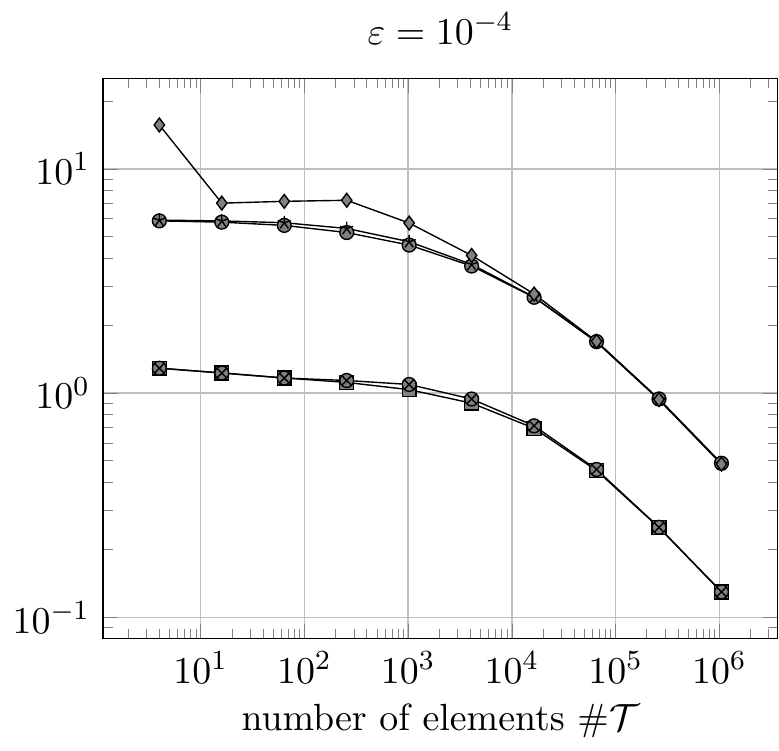}
  \includegraphics[width=0.49\textwidth]{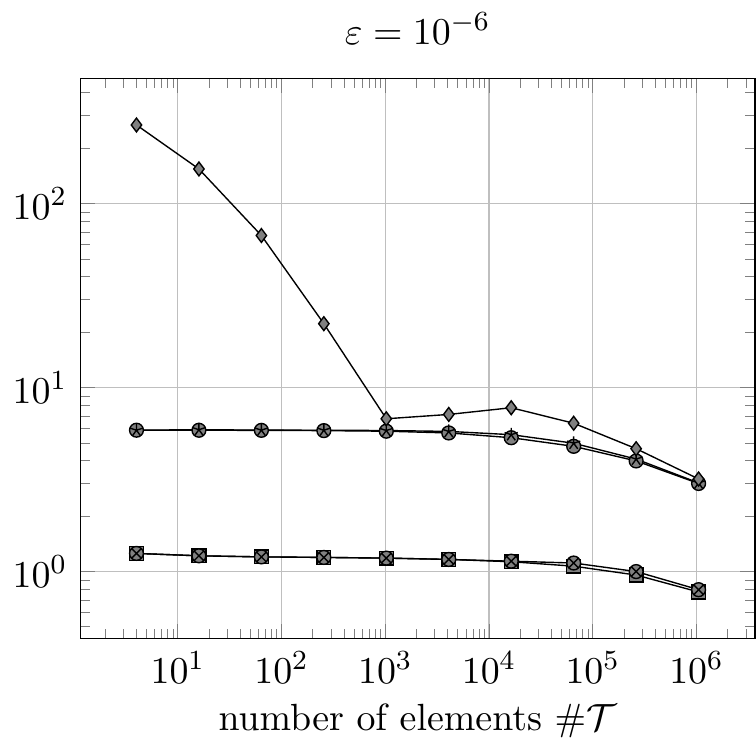}
  \includegraphics[width=0.49\textwidth]{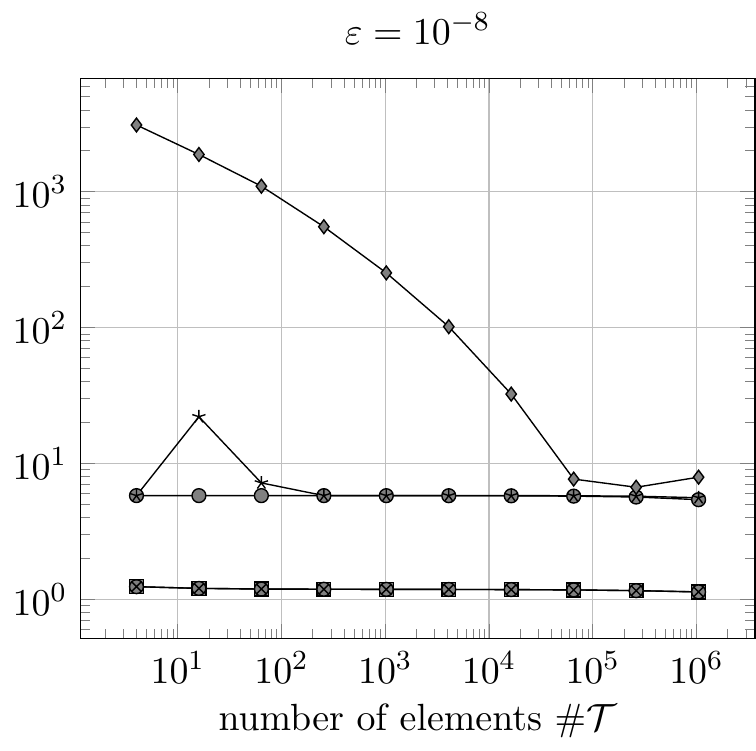}
  \caption{Comparison of the volume error and the error quantities for the skeleton variables on uniform meshes for the
    Example from Section~\ref{sec:num:smooth}.}
  \label{fig:unif}
\end{figure}
In Figure~\ref{fig:unif} we check 
the behaviour of skeleton variables compared to the volume error $\err_\Omega$
with error quantities $\err(\wat u^\star)$,
$\err(\wat\sigma^\star)$ for the skeleton variables $\wat u^\star$, $\wat\sigma^\star$
on a sequence of uniformly refined triangulations.
These quantities are defined for $\star\in\{a,b\}$ by
\begin{align*}
  \err(\wat u^\star) &:= \left( \norm{u-\widetilde u_h^\star}{L_2(\Omega)}^2
  + \eps^{1/2} \norm{\nabla(u-\widetilde u_h^\star)}{L_2(\Omega)}^2 \right)^{1/2}, \\
  \err(\wat\sigma^\star) &:= \left( \norm{\ssigma-\widetilde\ssigma_h^\star}{L_2(\Omega)}^2 
  + \eps \norm{\div(\ssigma-\widetilde\ssigma_h^\star)}{L_2(\Omega)}^2 \right)^{1/2}.
\end{align*}
Here, $\widetilde u_h^\star$ is the nodal interpolant of $\wat u_h^\star$ in $H^1(\Omega)$ and
$\widetilde\ssigma_h^\star$ is the lowest-order Raviart-Thomas projection of $\wat\sigma_h^\star$ in $\HH(\div,\Omega)$.
According to the definition of the skeleton norms~\eqref{Hpm}, there hold the estimates
\begin{align*}
  \norm{\wat u^\star - \wat u_h^\star}{1/2,\cS} \leq \err(\wat u^\star) \quad\text{and}\quad
  \norm{\wat \sigma^\star-\wat\sigma_h^\star}{-1/2,\cS} \leq \err(\wat\sigma^\star) \quad\text{for }
  \star\in\{a,b\}.
\end{align*}
From Figure~\ref{fig:unif} we observe that the errors of the trace variables $\wat u^\star$
and $\wat\sigma^a$ are comparable to the volume
error independently of the perturbation parameter $\eps$. 
Note that the analysis only predicts robust control of the norm of the error scaled with (postive) powers of $\eps$.
In contrast, for a fixed mesh, the errors in $\wat\sigma^b$ increase when $\eps$ gets smaller (on coarse
triangulations).
Note that our motivation was to robustly control field variables.
Skeleton variables appear due to the discontinuous variational formulation. They are not part of the
(original) problem and there is no need to control their approximation.
Of course, in the case of the interface variables (the skeleton variables $\wat u^a$ and $\wat\sigma^a$
restricted to $\Gamma$) we do need robust control since they are essential for the coupling with boundary
elements. Their robust control has been confirmed by the previous experiment.

\subsection{Singular solution}\label{sec:num:singular}
Let $\Omega$ denote the L-shaped domain sketched in Figure~\ref{fig:Lshape}.
In the interior we prescribe the solution
\begin{align*}
  u(x,y) := C_\eps I_{2/3}(r/\sqrt{\eps}) \cos(2/3\varphi),
\end{align*}
where $(r,\varphi)$ denote the polar coordinates of $(x,y)\in\Omega$, and $I_{\nu}$ with $\nu\in\R$ denotes the 
\emph{modified Bessel function of first kind}.
We choose $C_\eps$ such that $\norm{u}{L_\infty(\Omega)} = 1$.
Note that $u$ is harmonic in the sense that $\eps\Delta u - u = 0$ and has a singularity at the reentrant corner
$(x,y) = (0,0)$. 
Moreover, a closer look unveils that $u\in H^{5/3-s}(\Omega)$ for $s>0$. Hence, we expect suboptimal convergence
rates in the case of uniform refinement. Therefore, we stick to an adaptive strategy, where we choose
$\theta=\tfrac34$ in~\eqref{eq:markcrit}.
For the solution in the exterior we take the same function as in~\eqref{eq:smoothSol:ext} with constants $\alpha_1 =
0.125$ and $\alpha_2 = 0$.
\begin{figure}[htb]
  \includegraphics[width=0.6\textwidth]{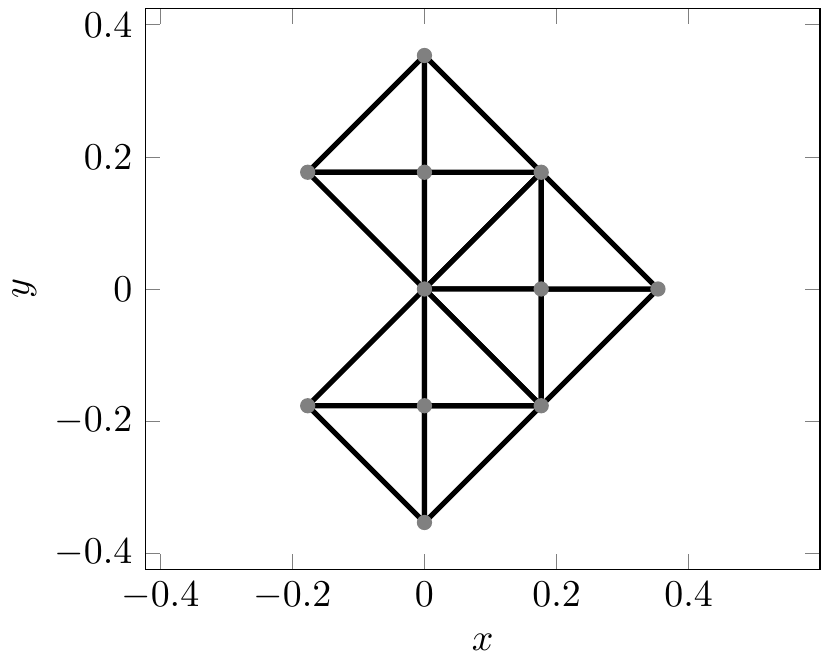}
  \caption{L-shaped domain and its initial triangulation $\TT_0$ with 12 volume elements.}
  \label{fig:Lshape}
\end{figure}
\begin{figure}[htb]
  \includegraphics[width=0.49\textwidth]{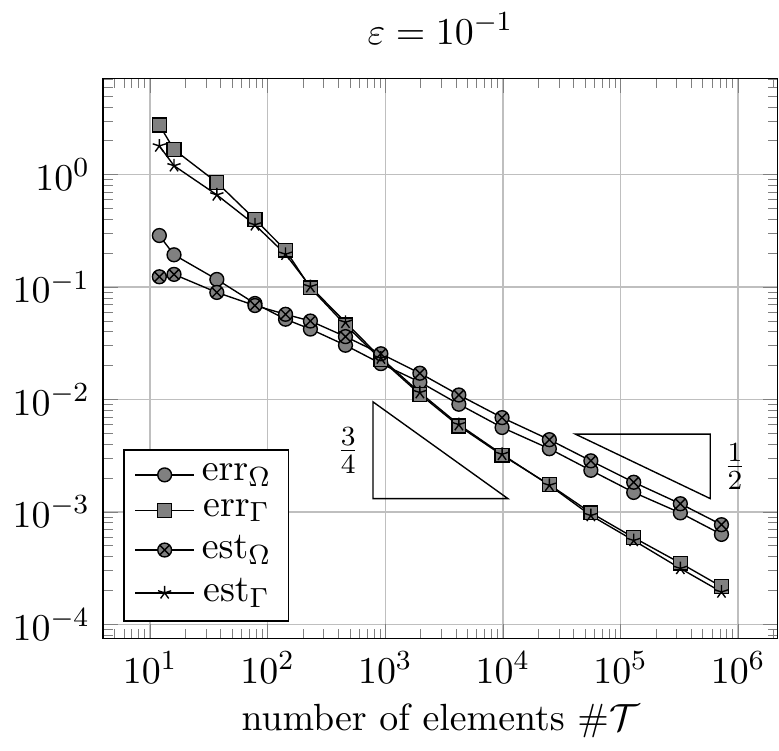}
  \includegraphics[width=0.49\textwidth]{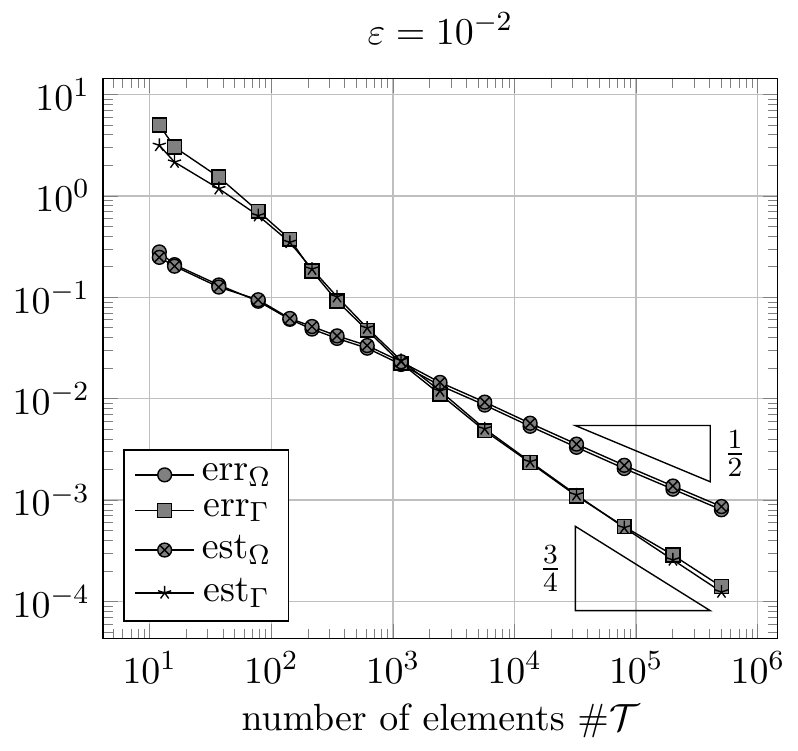}
  \includegraphics[width=0.49\textwidth]{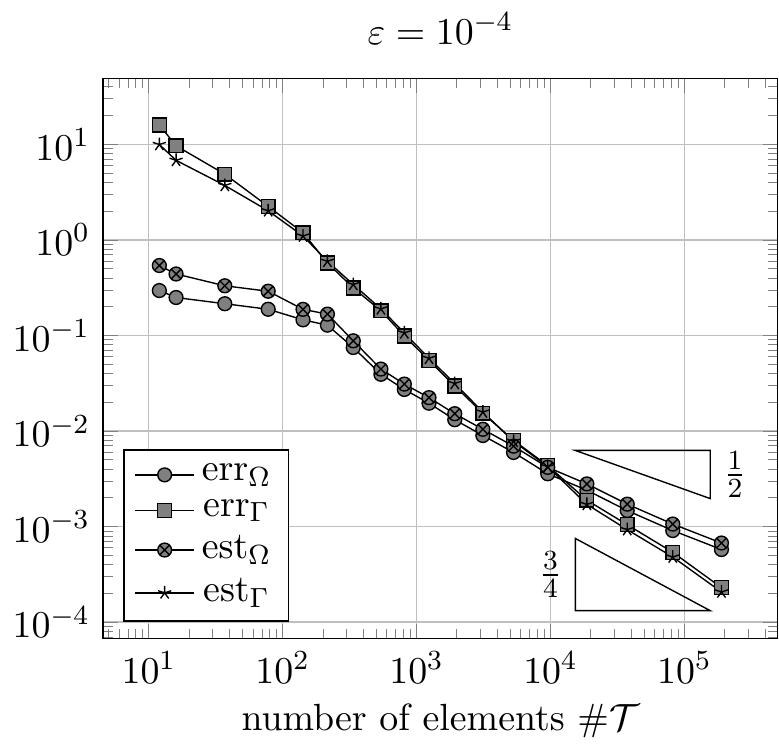}
  \includegraphics[width=0.49\textwidth]{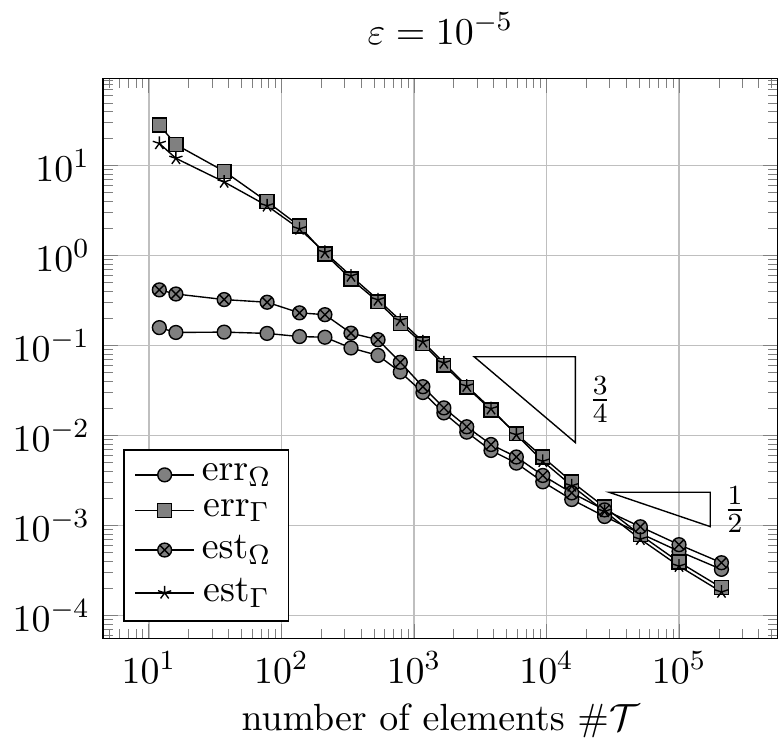}
  \caption{Error and estimator quantities for different values of $\eps$ for the example with singular solution
    (Section~\ref{sec:num:singular}).}
  \label{fig:singular}
\end{figure}

In Figure~\ref{fig:singular} we plot the quantities $\err_\Omega$, $\err_\Gamma$, $\est_\Omega$, and $\est_\Gamma$ for
$\eps=10^{-j}$ with $j\in\{1,2,4,5\}$.
We make the same observations as for the example from Section~\ref{sec:num:smooth}.
Furthermore, we stress that for $\eps$ close to $1$, the meshes are refined towards the singularity, whereas for
$0<\eps<10^{-2}$ the boundary layers get sharper, and the adaptive algorithm then resolves these layers.
The right plot in Figure~\ref{fig:ratio} shows the ratio $\err_\Omega/\est_\Omega$ of the volume error and the
volume estimator. Again we observe that after a few steps in the adaptive algorithm (after the layers are resolved) 
this ratio is close to one.

\subsection{Unknown solution}\label{sec:num:uk}
Again we choose $\Omega$ to be the L-shaped domain sketched in Figure~\ref{fig:Lshape}.
For this experiment we choose the data
\begin{align*}
 u_0 := \frac{1}2, \quad
  \phi_0 := 0, \quad
  f(x,y) := \begin{cases}
    1 & \text{if } (x-0.15)^2 + y^2 < \tfrac1{100} \\
    0 & \text{else}
  \end{cases}.
\end{align*}
\begin{figure}[htb]
  \includegraphics[width=0.75\textwidth]{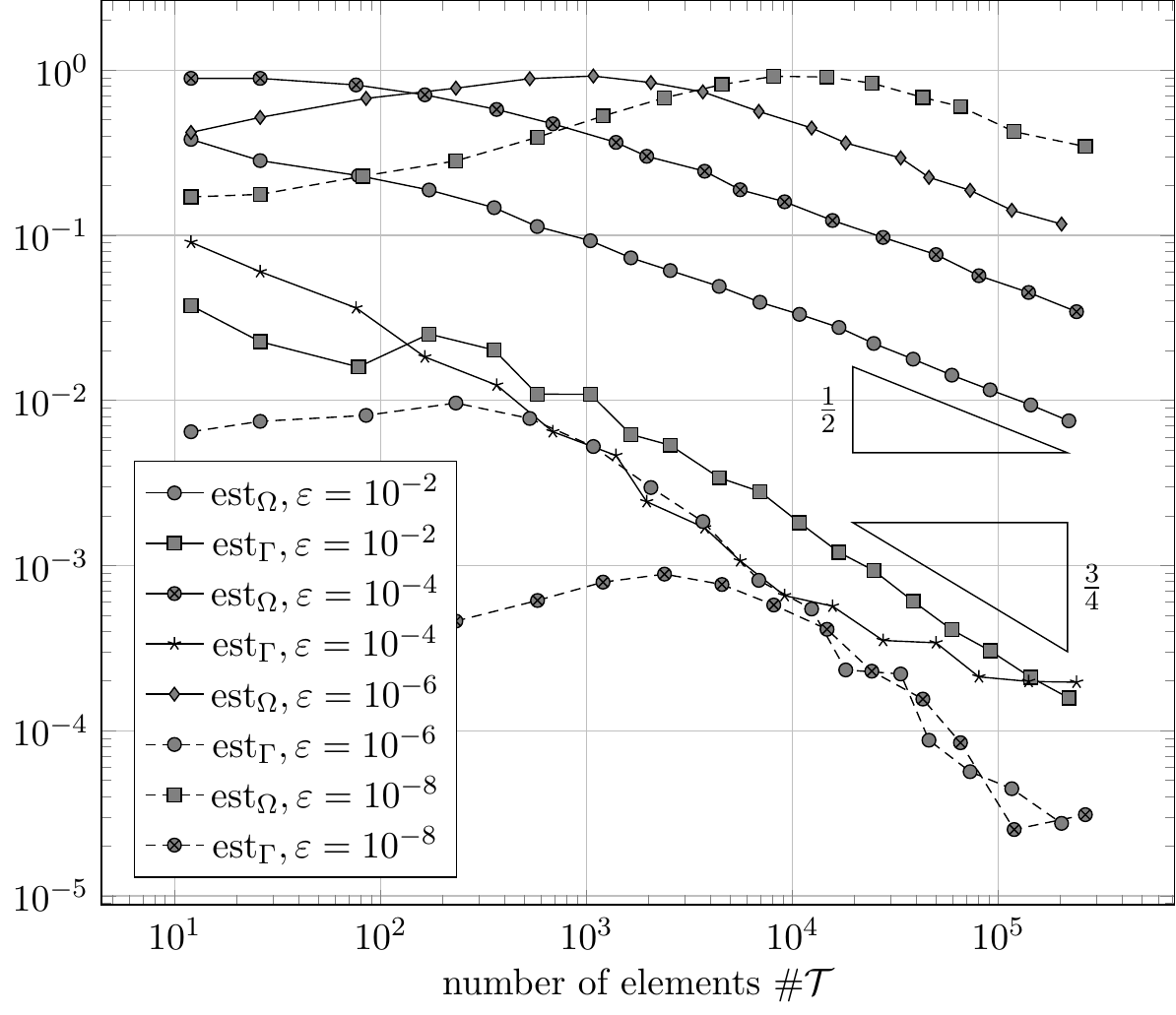}
  \caption{Error and estimator quantities for different values of $\eps$ for the example with unknown solution
    (Section~\ref{sec:num:uk}).}
  \label{fig:uk}
\end{figure}
In particular, for this choice of data we expect a boundary
layer (since $u_0\neq f$ on $\Gamma$) and a circular layer away from
the boundary. 
For this experiment we choose $\theta=\tfrac12$ in~\eqref{eq:markcrit}.

In Figure~\ref{fig:uk} we plot only the estimators $\est_\Omega$ and $\est_\Gamma$ since the exact solution is unknown.
As in the previous examples we observe that all estimator quantities converge with the optimal rate as soon as the layers
are resolved.

For illustration purposes, Figure~\ref{fig:uk:sol} shows the solution component $u_{hp}$ for different choices of $\eps$ on adaptively refined
meshes.
\begin{figure}[htb]
  \includegraphics[width=0.49\textwidth]{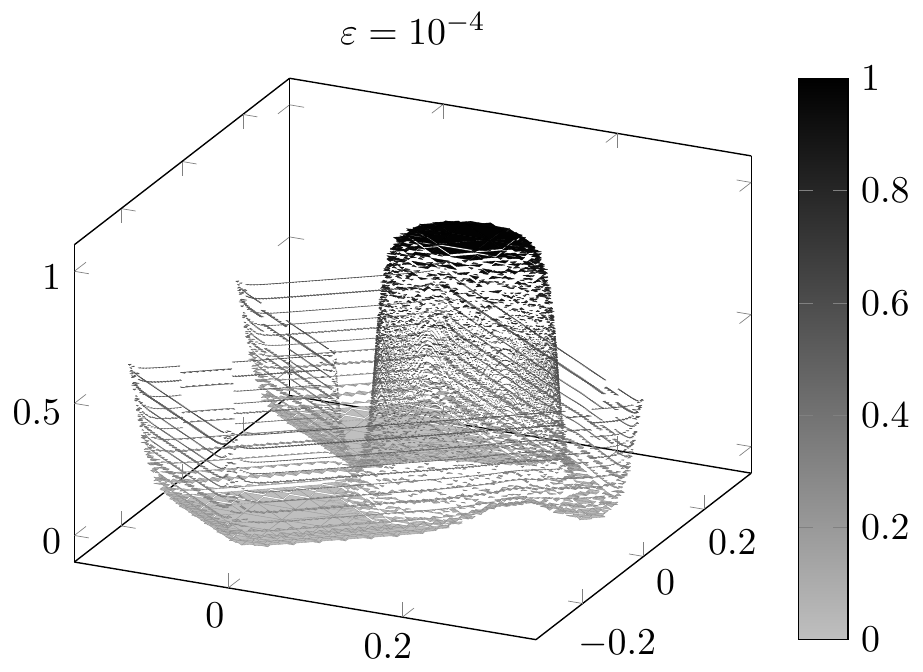}
  \includegraphics[width=0.49\textwidth]{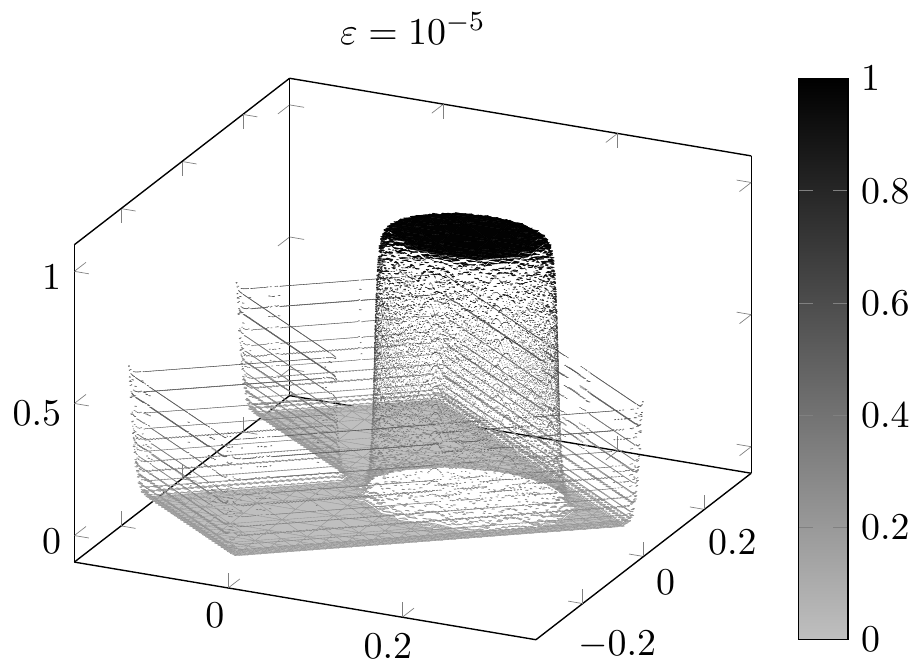}
  \includegraphics[width=0.49\textwidth]{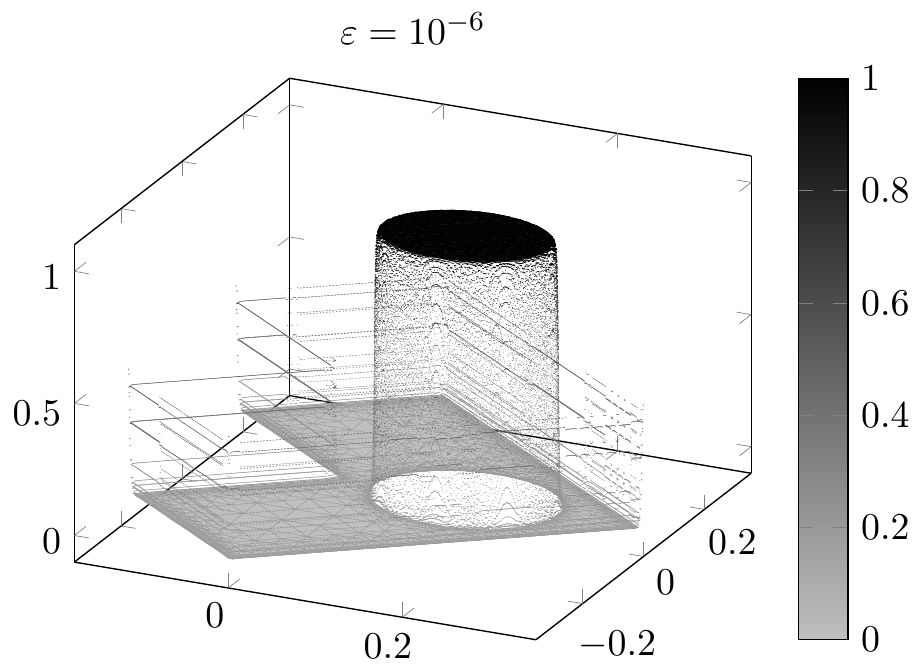}
  \includegraphics[width=0.49\textwidth]{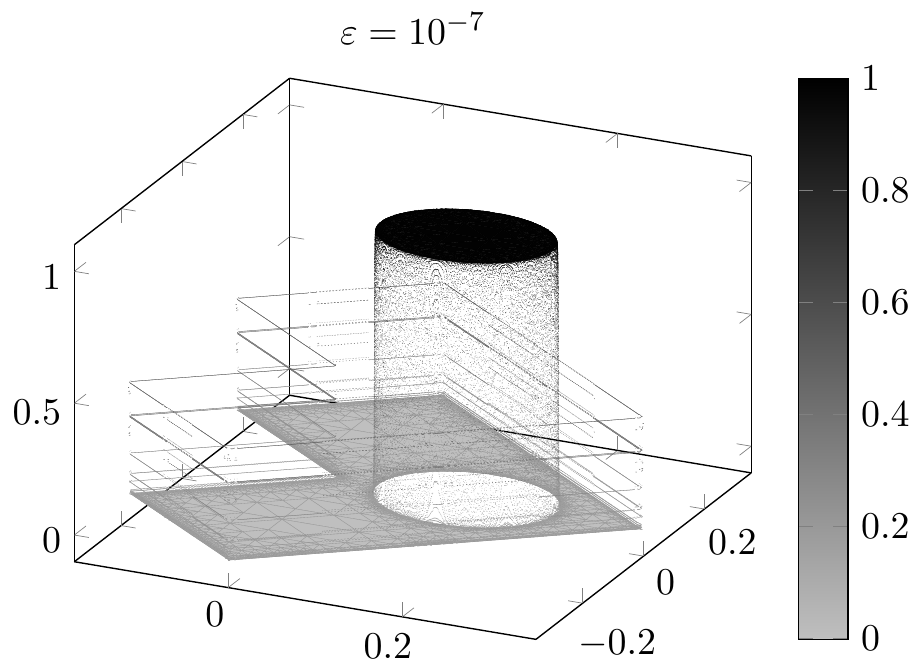}
  \caption{Solution component $u_{hp}$ for different values of $\eps$ for the example with unknown solution
  (Section~\ref{sec:num:uk}) on adaptively refined meshes with $15646$ ($\eps=10^{-4}$), $25384$ ($\eps=10^{-5}$), 
  $33550$ ($\eps=10^{-6}$), resp. $59581$ ($\eps=10^{-7}$) volume elements.}
  \label{fig:uk:sol}
\end{figure}

\bibliographystyle{abbrv}
\bibliography{bib}

\begin{thebibliography}{10}

\bibitem{AuradaEFFFGKMP_HMI}
M.~Aurada, M.~Ebner, M.~Feischl, S.~Ferraz-Leite, T.~F{\"u}hrer, P.~Goldenits,
  M.~Karkulik, M.~Mayr, and D.~Praetorius.
\newblock H{ILBERT}---a {MATLAB} implementation of adaptive 2{D}-{BEM}.
\newblock {\em Numer. Algorithms}, 67(1):1--32, 2014.

\bibitem{AuradaFFKMP_13_CFB}
M.~Aurada, M.~Feischl, T.~F\"uhrer, M.~Karkulik, J.~M. Melenk, and
  D.~Praetorius.
\newblock Classical {FEM}-{BEM} coupling methods: nonlinearities,
  well-posedness, and adaptivity.
\newblock {\em Comp. Mech.}, 51(4):399--419, 2013.

\bibitem{BielakM_91_SFE}
J.~Bielak and R.~C. MacCamy.
\newblock Symmetric finite element and boundary integral coupling methods for
  fluid-solid interaction.
\newblock {\em Quart. Appl. Math.}, 49(1):107--119, 1991.

\bibitem{BroersenS_14_RPG}
D.~Broersen and R.~Stevenson.
\newblock A robust {Petrov}-{Galerkin} discretisation of convection-diffusion
  equations.
\newblock {\em Comput. Math. Appl.}, 68(11):1605--1618, 2014.

\bibitem{BroersenS_15_PGD}
D.~Broersen and R.~Stevenson.
\newblock A {Petrov}-{Galerkin} discretization with optimal test space of a
  mild-weak formulation of convection-diffusion equations in mixed form.
\newblock {\em IMA J. Numer. Anal.}, 35(1):39--73, 2015.

\bibitem{ChanHBTD_14_RDM}
J.~Chan, N.~Heuer, T.~Bui-Thanh, and L.~Demkowicz.
\newblock Robust {DPG} method for convection-dominated diffusion problems {II}:
  {Adjoint} boundary conditions and mesh-dependent test norms.
\newblock {\em Comput. Math. Appl.}, 67(4):771--795, 2014.

\bibitem{Costabel_88_BIO}
M.~Costabel.
\newblock Boundary integral operators on {Lipschitz} domains: Elementary
  results.
\newblock {\em SIAM J. Math. Anal.}, 19:613--626, 1988.

\bibitem{Costabel_88_SMC}
M.~Costabel.
\newblock A symmetric method for the coupling of finite elements and boundary
  elements.
\newblock In J.~R. Whiteman, editor, {\em The Mathematics of Finite Elements
  and Applications VI}, pages 281--288, London, 1988. Academic Press.

\bibitem{CostabelS_88_CFE}
M.~Costabel and E.~P. Stephan.
\newblock Coupling of finite elements and boundary elements for inhomogeneous
  transmission problems in {$\R^3$}.
\newblock In J.~R. Whiteman, editor, {\em The Mathematics of Finite Elements
  and Applications VI}, pages 289--296, London, 1988. Academic Press.

\bibitem{DemkowiczG_11_ADM}
L.~Demkowicz and J.~Gopalakrishnan.
\newblock Analysis of the {DPG} method for the {Poisson} problem.
\newblock {\em SIAM J. Numer. Anal.}, 49(5):1788--1809, 2011.

\bibitem{DemkowiczG_11_CDP}
L.~Demkowicz and J.~Gopalakrishnan.
\newblock A class of discontinuous {Petrov-Galerkin} methods. {Part II}:
  {O}ptimal test functions.
\newblock {\em Numer. Methods Partial Differential Eq.}, 27:70--105, 2011.

\bibitem{DemkowiczH_13_RDM}
L.~Demkowicz and N.~Heuer.
\newblock Robust {DPG} method for convection-dominated diffusion problems.
\newblock {\em SIAM J. Numer. Anal.}, 51(5):2514--2537, 2013.

\bibitem{FuehrerHK_CDB}
T.~F{\"u}hrer, N.~Heuer, and M.~Karkulik.
\newblock On the coupling of {DPG} and {BEM}.
\newblock {http://arXiv.org/abs/1508.00630}, 2016.
\newblock Accepted for publication in {\em Math. Comp.}

\bibitem{GopalakrishnanQ_14_APD}
J.~Gopalakrishnan and W.~Qiu.
\newblock An analysis of the practical {DPG} method.
\newblock {\em Math. Comp.}, 83(286):537--552, 2014.

\bibitem{HeuerK_15_DPG}
N.~Heuer and M.~Karkulik.
\newblock {DPG} method with optimal test functions for a transmission problem.
\newblock {\em Comput. Math. Appl.}, 70(5):1504--1518, 2015.

\bibitem{HeuerK_RDM}
N.~Heuer and M.~Karkulik.
\newblock A robust {DPG} method for singularly perturbed reaction-diffusion
  problems.
\newblock {http://arXiv.org/abs/1509.07560}, 2015.

\bibitem{HsiaoW_08_BIE}
G.~C. Hsiao and W.~L. Wendland.
\newblock {\em Boundary Integral Equations}.
\newblock Springer, 2008.

\bibitem{JohnsonN_80_CBI}
C.~Johnson and J.-C. N\'ed\'elec.
\newblock On the coupling of boundary integral and finite element methods.
\newblock {\em Math. Comp.}, 35:1063--1079, 1980.

\bibitem{LiN_98_UCF}
J.~Li and I.~M. Navon.
\newblock Uniformly convergent finite element methods for singularly perturbed
  elliptic boundary value problems. {I}. {R}eaction-diffusion type.
\newblock {\em Comput. Math. Appl.}, 35(3):57--70, 1998.

\bibitem{LinS_12_BFE}
R.~Lin and M.~Stynes.
\newblock A balanced finite element method for singularly perturbed
  reaction-diffusion problems.
\newblock {\em SIAM J. Numer. Anal.}, 50(5):2729--2743, 2012.

\bibitem{Linss_10_LAM}
T.~Lin{\ss}.
\newblock {\em Layer-adapted meshes for reaction-convection-diffusion
  problems}, volume 1985 of {\em Lecture Notes in Mathematics}.
\newblock Springer-Verlag, Berlin, 2010.

\bibitem{MaghnoujiN_06_BLT}
A.~Maghnouji and S.~Nicaise.
\newblock Boundary layers for transmission problems with singularities.
\newblock {\em Electron. J. Differential Equations}, pages No. 14, 16 pp.
  (electronic), 2006.

\bibitem{McLean_00_SES}
W.~McLean.
\newblock {\em Strongly Elliptic Systems and Boundary Integral Equations}.
\newblock Cambridge University Press, 2000.

\bibitem{MelenkX_SPhpFEM}
J.~M. Melenk and C.~Xenophontos.
\newblock Robust exponential convergence of {$hp$}-{FEM} in balanced norms for
  singularly perturbed reaction-diffusion equations.
\newblock {\em Calcolo}, 53(1):105--132, 2016.

\bibitem{NicaiseX_09_FEM}
S.~Nicaise and C.~Xenophontos.
\newblock Finite element methods for a singularly perturbed transmission
  problem.
\newblock {\em J. Numer. Math.}, 17(4):245--275, 2009.

\bibitem{NicaiseX_13_Chp}
S.~Nicaise and C.~Xenophontos.
\newblock Convergence analysis of an {$hp$} finite element method for
  singularly perturbed transmission problems in smooth domains.
\newblock {\em Numer. Methods Partial Differential Equations},
  29(6):2107--2132, 2013.

\bibitem{RoosS_15_CSB}
H.-G. Roos and M.~Schopf.
\newblock Convergence and stability in balanced norms of finite element methods
  on {S}hishkin meshes for reaction-diffusion problems.
\newblock {\em ZAMM Z. Angew. Math. Mech.}, 95(6):551--565, 2015.

\bibitem{RoosST_08_RNM}
H.-G. Roos, M.~Stynes, and L.~Tobiska.
\newblock {\em Robust numerical methods for singularly perturbed differential
  equations}, volume~24 of {\em Springer Series in Computational Mathematics}.
\newblock Springer-Verlag, Berlin, second edition, 2008.
\newblock Convection-diffusion-reaction and flow problems.

\bibitem{Sayas_09_VJN}
F.-J. Sayas.
\newblock The validity of {J}ohnson-{N}\'ed\'elec's {BEM}-{FEM} coupling on
  polygonal interfaces.
\newblock {\em SIAM J. Numer. Anal.}, 47(5):3451--3463, 2009.

\bibitem{SolinSD_04_HOF}
P.~{\v{S}}ol{\'{\i}}n, K.~Segeth, and I.~Dole{\v{z}}el.
\newblock {\em Higher-order finite element methods}.
\newblock Studies in Advanced Mathematics. Chapman \& Hall/CRC, Boca Raton, FL,
  2004.

\bibitem{XenophontosF_03_UAS}
C.~Xenophontos and S.~R. Fulton.
\newblock Uniform approximation of singularly perturbed reaction-diffusion
  problems by the finite element method on a {S}hishkin mesh.
\newblock {\em Numer. Methods Partial Differential Equations}, 19(1):89--111,
  2003.

\end{thebibliography}
\end{document}